\documentclass[10pt]{article}
\usepackage{amsthm}
\usepackage{amsmath}
\usepackage{bbm}
\usepackage{epsfig}
\usepackage{graphicx}
\usepackage{mathrsfs}
\usepackage{amsthm}
\usepackage{enumerate}
\usepackage{amsfonts}
\usepackage{color}
\usepackage{rotating}
\usepackage{epstopdf}
\usepackage{verbatim}
\usepackage{hyperref}
\usepackage[authoryear,round,comma]{natbib}
\usepackage{amsfonts}
\usepackage{amssymb}
\usepackage{appendix}
\usepackage{amsthm}
\usepackage{footnote}

\DeclareGraphicsExtensions{.eps,.png,.jpg,.pdf}

\newtheorem{lemma}{Lemma}[section]
\newtheorem{proposition}{Proposition}[section]
\newtheorem{theorem}{Theorem}[section]
\newtheorem{corollary}{Corollary}[section]

\theoremstyle{remark}
\newtheorem{remark}{\bf{Remark}}[section]
\newtheorem{example}{\bf{Example}}[section]

\newcommand{\ConvD}{\overset{d}{\rightarrow}}
\newcommand{\ConvP}{\overset{p}{\rightarrow}}

\newcommand{\Cov}{\mathrm{Cov}}
\newcommand{\Corr}{\mathrm{Corr}}
\newcommand{\Var}{\mathrm{Var}}
\newcommand{\E}{\mathbb{E}}

\newcommand{\I}{\mathrm{I}}

\makeatletter
\def\blfootnote{\xdef\@thefnmark{}\@footnotetext}
\makeatother

\title{A unified approach to self-normalized block sampling}
\author{Shuyang Bai, Murad S. Taqqu and Ting Zhang\footnote{The authors are listed in alphabetical order. They all contributed equally to this work.}}
\begin{document}
\maketitle


\begin{abstract}
The inference procedure for the mean of a stationary time series is usually quite different under various model assumptions because the partial sum process behaves differently depending on whether the time series is short or long-range dependent, or whether it has a light or heavy-tailed marginal distribution. In the current paper, we develop an asymptotic theory for the self-normalized block sampling, and prove that the corresponding block sampling method can provide a unified inference approach for the aforementioned different situations in the sense that it does not require the {\em a priori} estimation of auxiliary parameters. Monte Carlo simulations are presented to illustrate its finite-sample performance. The \texttt{R} function implementing the method is available from the authors.
\end{abstract}

\blfootnote{
\begin{flushleft}
AMS 2000 subject classifications: Primary 62G09; secondary 60G18\\
Key words and phrases: time series; subsampling; block sampling; sampling window; self normalization; heavy tails; long-range dependence; long memory
\end{flushleft}
}

\section{Introduction}\label{sec:intro}
Given samples $X_1,\ldots,X_n$ from a stationary process $\{X_i\}_{i \in \mathbb Z}$ with mean $\mu = E(X_0)$, the sample average $\bar X_n = n^{-1} \sum_{i=1}^n X_i$ serves as a natural estimator for the population mean $\mu$. To conduct statistical inference on the mean $\mu$ such as hypothesis testing or the construction of confidence intervals, one needs an asymptotic theory on the sample average for dependent data. The development of such a theory has been an active area of research. Consider first the classical case, where by assuming certain short-range dependence conditions, one obtains the usual central limit theorem, that is,
\begin{equation}\label{eqn:CLTmu}
n^{1/2} (\bar X_n - \mu) \ConvD N(0, \sigma^2),
\end{equation}
where $\ConvD$ denotes the convergence in distribution, and $\sigma^2$ is the long-run variance which typically is the sum of autocovariances of all orders. The short-range dependence conditions mentioned above include, but are not limited to, the $m$-dependence condition of \citet{Hoeffding:Robbins:1948}, the strong mixing condition of \citet{Rosenblatt:1956} and its variants, and the $p$-stability condition based on functional dependence measures of \citet{Wu:2005}; see also \citet{Ibragimov:Linnik:1971}, \citet{Peligrad:1996}, \citet{Maxwell:Woodroofe:2000}, \citet{Bradley:2007}, \citet{Wu:2011} and references therein. Once one has (\ref{eqn:CLTmu}), an asymptotic $100(1-\alpha)\%$ confidence interval of $\mu$ can be constructed as
\begin{equation}\label{eqn:CImu}
[\bar X_n -  n^{-1/2} \sigma q_{1-\alpha/2},\  \bar X_n  + n^{-1/2} \sigma q_{1-\alpha/2}]
\end{equation}
where $q_{1-\alpha/2}$ is the $(1-\alpha/2)$-th quantile of the standard normal distribution. However, the implementation of (\ref{eqn:CImu}) requires the estimation of a nuisance parameter $\sigma$,  which can itself be a challenging problem and often relies on techniques including tapering and thresholding to achieve  consistency; see for example \citet{Newey:West:1987}, \citet{Flegal:Jones:2010}, \citet{Politis:2011} and \citet{Zhang:Wu:2012} among others.

If the process $(X_i)_{i\in \mathbb{Z}}$ is   heavy-tailed (distributional tail behaving like $x^{-\alpha}$ with $\alpha\in (1,2)$) so that the variance is infinite, one typically has
\begin{equation}\label{eqn:CLTmu heavy}
n^{1-1/\alpha}\ell(n)^{-1} (\bar X_n - \mu) \ConvD S_\alpha(\sigma,\beta,0),
\end{equation}
where $\ell(n)$ is a slowly varying function satisfying $\lim_{n \to \infty} \ell(an)/\ell(n) = 1$ for any $a > 0$, and $S_\alpha(\sigma,\beta,0)$ is the centered $\alpha$-stable random variable with scale parameter $\sigma>0$ and skewness parameter $\beta\in[-1,1]$.  We refer the reader to the monographs by \citet{samorodnitsky:taqqu:1994:stable}, \citet{nolan:2015:stable} and \citet{resnick:2007:heavy} for an introduction. See  also \citet{feldman:1998:practical} for examples of heavy tails from finance, signal processing,  networks, etc. Here the use of (\ref{eqn:CLTmu heavy}) for constructing confidence interval as in (\ref{eqn:CImu}) becomes more difficult due to additional unknown parameters $\sigma$, $\alpha$ and $\beta$,  as well as the unknown $\ell(n)$.

There has been a considerable amount of research focusing on the situation where the short-range dependence condition fails, and processes with \emph{long-range dependence} (also called  ``long memory" or ``strong dependence") has attracted a lot of attention in various fields including econometrics, finance, hydrology and telecommunication among others; see for example  \cite{mandelbrot:1968:noah}, \citet{ding:1993:long},  \citet{leland:1994:self} and \citet{baillie:1996:long}. We also refer the reader to the monographs by \citet{doukhan:oppenheim:taqqu:2003:theory}, \citet{giraitis:koul:surgailis:2009:large} and \citet{beran:2013:long} for an introduction. For long-range dependent processes, it may be established that
\begin{equation}\label{eqn:CLTLRDmu}
n^{1-H} \ell(n)^{-1} (\bar X_n - \mu) \ConvD Y,
\end{equation}
where $H \in (1/2,1)$ is the Hurst index (or the long memory index), $\ell(n)$ is a slowly varying function, and $Y$ is typically a random variable which can be expressed by a multiple Wiener-It\^{o} integral and is not necessarily Gaussian. The large sample theory of the form (\ref{eqn:CLTLRDmu}) has been studied by  \citet{Davydov:1970}, \citet{Taqqu:1975}, \citet{Dobrushin:Major:1979}, \citet{Avram:Taqqu:1987}, \citet{Ho:Hsing:1997}, \citet{Wu:2006} and \citet{Bai:Taqqu:2014} among others. Therefore, the asymptotic behavior of the sample average and thus the inference procedure can become very different for long-range dependent processes, and the convergence rate in (\ref{eqn:CLTLRDmu}) depends critically on the Hurst index $H$ which characterizes the dependence strength. Hence, in order to apply (\ref{eqn:CLTLRDmu}) for inference, unlike the case with short-range dependence and light tail, one needs to estimate in addition the Hurst index $H$ and possibly the slowly varying function $\ell(n)$, which can be quite nontrivial. Furthermore, the distribution of a non-Gaussian $Y$ (which also depends on $H$) has not been numerically evaluated in general. For the special case of the Rosenblatt distribution where it is evaluated, see \citet{veillette:taqqu:2013:properties}.

There has recently been a surge of attention in using some random normalizers to avoid, or reduce the number of  nuisance parameters that need to be estimated for statistical inference. For example, \citet{McElroy:Politis:2002} considered using the sample standard deviation as the normalizer for inference on the mean of heavy-tailed linear processes that satisfy the strong mixing condition; see also \citet{Romano:Wolf:1999} for the use of a similar normalizer for independent observations. \citet{Lobato:2001}, \citet{Shao:2010}, \citet{Zhou:Shao:2013} and \citet{Huang:Volgushev:Shao:2015} used a normalization of the type
\begin{equation}\label{eqn:normalizer}
D_n = \left\{n^{-1}\sum_{k=1}^n \left(\sum_{i=1}^k X_i -\frac {k}{n} \sum_{i=1}^n X_i\right)^2\right\}^{1/2}
\end{equation}
for finite-variance short-range dependent time series. \citet{Fan:2010} used the normalizer $D_n$ for long-range dependent time series with finite variances. Results have also been obtained by \citet{McElroy:Politis:2013} using a lag-window normalizer instead of $D_n$ in (\ref{eqn:normalizer}). \citet{McElroy:Politis:2007}, moreover, considered the following non-centered stochastic volatility model $X_i = \mu + \sigma_i Z_i$, $i \geq 1$, where $\{\sigma_i\}$ and $\{Z_i\}$ are independent, $\{\sigma_i\}$ is i.i.d.\ heavy-tailed and $\{Z_i\}$ is a Gaussian process.  They proposed to use a random normalizer involving two terms that account for heavy-tailedness and long memory respectively. The term in their normalizer which accounts for long memory requires the choice of an additional tuning parameter. Therefore, it seems that the specific form of the normalization depends critically on the particular time series that is being considered, and different normalizers have been used in the literature to account for the heavy-tail and/or long-range dependent characteristics of the time series.

The current paper aims to provide a unified inference procedure by adopting the normalizer $D_n$ in (\ref{eqn:normalizer}) and developing an asymptotic theory using self-normalized block sums.
As observed by \citet{Shao:2011:CPLRD}, self-normalization itself is not able to fully avoid the problem of estimating the nuisance parameters, as the asymptotic distribution  at least  depends on  the unknown Hurst index $H$ for long-range dependent processes. In order to provide a unified approach that does not rely on the estimation of any nuisance parameter to determine the strength of dependence or heavy-tailedness, certain nonparametric techniques such as the block sampling\footnote{The following terms are used interchangeably in the literature: block sampling, subsampling, sampling window method.} must be utilized to obtain the asymptotic quantiles. However, this requires developing an asymptotic theory on the self-normalized block sums for a general class of processes. This task may be nontrivial if we want it to include processes with long-range dependence and/or heavy-tails. Block sampling has been mainly studied in the literature in the {\em non-self-normalized} setting, where the normalizer converges in probability to a nonzero constant, thus simplifying the proof; see for example \citet{Hall:Jing:Lahiri:1998} for nonlinear transforms of Gaussian processes, \citet{Nordman:Lahiri:2005} for linear processes, and \citet{Zhang:Ho:Wendler:Wu:2013} for nonlinear transforms of linear processes. 
\citet{jach:2012:subsampling} applied block sampling to the model $X_i = \mu + \sigma_i Z_i$, $i \geq 1$, considered by \citet{McElroy:Politis:2007} but with $Z_i$ replaced by $g(Z_i)$ where $g$ is a possibly nonlinear function with Hermite rank one.
For more information on block sampling, see  \citet{sherman:carlstein:1996:replicate} and \citet{lahiri:2003:resampling}. \citet{betken:Wendler:2015:subsampling} recently obtained interesting results in the context of long-range dependence. They are briefly discussed in Section \ref{sec:ass A3} (see (\ref{eq:BW b_n}) below).

The current paper considers \emph{self-normalized} block sums using $D_n$ in (\ref{eqn:normalizer}) as normalizer. As observed by \citet{Fan:2010}, the development of an asymptotic theory in this case can be very nontrivial even for Gaussian processes. Developing a rigorous proof is stated as an open problem. The goal of this paper is to develop such a proof for nonlinear functions of Gaussian processes with either short or long-range dependence, and including heavy-tails.

The remaining of the paper is organized as follows. Section \ref{sec:method} introduces the self-normalized block sampling (SNBS) method, whose asymptotic theory is established in Section \ref{sec:theory}.  Section \ref{sec:example} contains examples. Monte Carlo simulations are carried out in Section \ref{sec:numericalexperiments} to examine the finite-sample performance of the method.

\section{Self-Normalized Block Sampling}\label{sec:method}
Let $X_1,\ldots,X_n$ be observations from a stationary process $(X_i)_{i \in \mathbb Z}$ with mean $\mu = E(X_0)$, and denote by
$
S_{j,k} = \sum_{i=j}^k X_i,  j \leq k,
$
its partial sums from $j$ to $k$. Of particular interest is $S_{1,n}=\sum_{i=1}^n X_i$. We propose using the self-normalized quantity
\begin{equation}\label{eqn:Tn}
T_n^*  ={S_{1,n} - n\mu \over D_n}
\end{equation}
for making statistical inference on the mean $\mu$, where $D_n$, defined in (\ref{eqn:normalizer}), can now be written
\begin{equation}\label{eq:D_n}
D_n=\left\{n^{-1}\sum_{k=1}^n \left(S_{1,k}-\frac
{k}{n}S_{1,n}\right)^2\right\}^{1/2}.
\end{equation}
In order to make inference on $\mu$, we  need to know the distribution $P(T_n^*\le x)$.

A first idea is to use the asymptotic distribution of (\ref{eqn:Tn}). This would require knowing the  weak limit of the normalized partial sum process,
namely,
\begin{equation}\label{eqn:general limit}
\{n^{-H}\ell(n)^{-1}(S_{\lfloor nt \rfloor} - n\mu),\ 0 \leq t \leq 1\} \Rightarrow \{Y(t),\ 0 \leq t \leq 1\},
\end{equation}
where  $t \in [0,1]$,  $\lfloor nt \rfloor$ denotes the largest integer not exceeding $nt$, and  $\Rightarrow$ denotes weak convergence in Skorokhod space with suitable topology.
By \citet{lamperti:1962:semi}, if (\ref{eqn:general limit}) holds, then the process $Y(t)$ is self-similar with stationary increments, with Hurst index\footnote{We exclude the degenerate case $H=1$.}  $0<H< 1$($H$-sssi), and with $\ell(\cdot)$ a slowly varying function. Recall that a process $Y(t)$ is said to be self-similar with Hurst index $H$ if $\{Y(c t),~t\ge 0\}$ has the same finite-dimensional distributions as $\{c^{H}Y(t),~t\ge 0\}$, for any $c>0$.

The most important example of (\ref{eqn:general limit}) is when $(X_i)_{i \in \mathbb Z}$ is short-range dependent and admits finite variance, in which case one  expects
\begin{equation}\label{eqn:SntSRD}
\{n^{-1/2}(S_{\lfloor nt \rfloor} - n\mu),\ 0 \leq t \leq 1\} \Rightarrow \{\sigma B(t),\ 0 \leq t \leq 1\},
\end{equation}
where $B(\cdot)$ is the standard Brownian motion, and $\sigma^2>0$ is the  long-run variance; see for example, the invariance principle of \citet{herrndorf:1984:functional} under strong mixing, and also the strong invariance principle of \citet{Wu:2007}.
When $\{X_i\}$ is short-range dependent but has infinite  variance with distributional tail regularly varying of order $-\alpha$ where $\alpha\in (1,2)$,  one has typically
\begin{equation}\label{eqn:SntSRDStable}
\{n^{-1/\alpha}\ell(n)^{-1}(S_{\lfloor nt \rfloor} - n\mu),\ 0 \leq t \leq 1\} \Rightarrow \{L_{\alpha,\sigma,\beta}(t),\ 0 \leq t \leq 1\},
\end{equation}
where $L_{\alpha,\sigma,\beta}(t)$ is a centered $\alpha$-stable L\'evy process with scale parameter $\sigma>0$ and skewness parameter $\beta\in [-1,1]$.  See, for example, \citet{skorokhod:1957:limit}, \citet{avram:taqqu:1992:weak}, \citet{tyran:2010:convergence}, \citet{tyran:2010:functional} and \citet{basrak:2012:functional} for the specification of the corresponding Skorohod topology.

Under long-range dependence, the limit in (\ref{eqn:general limit}) can be quite complicated. A typical class of  convergence in this case is
\begin{equation}\label{eqn:SntLRD}
\{n^{-H}\ell(n)^{-1}(S_{\lfloor nt \rfloor} - n\mu),\ 0 \leq t \leq 1\} \Rightarrow \{ c  Z_{m,H}(t),\ 0 \leq t \leq 1\},
\end{equation}
where $1/2<H<1$, $ Z_{m,H}(\cdot)$ is the $m$-th order Hermite process  which can be expressed by a multiple Wiener-It\^o integral (see, e.g.,  \citet{Dobrushin:Major:1979} and \citet{taqqu:1979:convergence}), and $c$ is a constant depending on $H$, $m$ and $\ell(n)$. A Hermite process $Z_{m,H}(\cdot)$ with $m\ge 2$ is non-Gaussian, and
when $m=1$ it is the Gaussian process called fractional Brownian motion, also denoted by $B_H(\cdot)$. One can also consider the anti-persistent case $H<1/2$, where the limit can be more complicated than $ Z_{m,H}(\cdot)$ (see \citet{major:1981:limit}).

Applying  the \emph{same} normalization $n^{-H}\ell(n)^{-1}$ to both the numerator and denominator of $T_n^*$ in (\ref{eqn:Tn}), one can establish  as in \citet{Lobato:2001}, via (\ref{eqn:general limit}) and the Continuous Mapping Theorem that as $n\rightarrow\infty$,
\begin{equation}\label{eqn:general self norm limit}
T_n^*={n^{-H}\ell(n)^{-1}(S_{1,n} - n\mu) \over n^{-H}\ell(n)^{-1}\left\{n^{-1}\sum_{k=1}^n (S_{1,k}-\frac
{k}{n}S_{1,n})^2\right\}^{1/2}} ~\ConvD ~ T:=\frac{Y(1)}{D},
\end{equation}
with
\begin{equation}\label{eq:D}
D=\left[\int_0^1 \{Y(s) - s Y(1)\}^2ds\right]^{1/2}.
\end{equation}
Note that $D>0$ almost surely. Indeed, if $P(D=0)>0$, then with positive probability $Y(s)=sY(1)$, which has locally bounded variation. This cannot happen  by Theorem 3.3 of \citet{vervaat:1985:sample}, since we assume $H<1$.

In particular, in the short-range dependent case (\ref{eqn:SntSRD}), one gets
\begin{equation*}
T_n^* \ConvD {B(1) \over \left[\int_0^1 \{B(s) - sB(1)\}^2ds\right]^{1/2}},
\end{equation*}
where the limit does not depend on any nuisance parameter. However, this nice property no longer holds in the other cases (\ref{eqn:SntSRDStable}) and (\ref{eqn:SntLRD}), since  $Y(t)$ in either case involves additional parameters. Therefore, except for short-range dependent light-tailed processes, self-normalization itself is usually not able to fully avoid the problem of estimating the nuisance parameters, and we shall  follow here \citet{Hall:Jing:Lahiri:1998} and consider a block sampling approach. See also Chapter 5 of \citet{politis:1999:subsampling}. Let
\begin{equation}\label{eqn:TiBn}
T_{i,b_n}^* = {S_{i,i+b_n-1} - b_n \mu \over \sqrt{b_n^{-1}\sum_{k=i}^{i+b_n-1}(S_{i,k} - b_n^{-1}(k-i+1) S_{i,i+b_n-1})^2}}=:\frac{S_{i,i+b_n-1} - b_n \mu}{D_{i,b_n}},\quad 1\le i \le n-b_n+1,
\end{equation}
which is the block version of $T_n^*$ in (\ref{eqn:Tn}) for the subsample $X_i,\ldots,X_{i+b_n-1}$, where $b_n$ denotes the block size. Observe that there is a considerable overlap between successive blocks, since as $i$ increases to $i+1$, the subsample becomes $X_{i+1},\ldots,X_{i+b_n}$, and thus includes many of the same observations.

 We consider using the empirical distribution function
\begin{equation}\label{eq:F_N}
\widehat{F}_{n,b_n}^*(x) = {1 \over n-b_n+1} \sum_{i=1}^{n-b_n+1} \mathrm I(T_{i,b_n}^* \leq x),
\end{equation}
where $\I(\cdot) $ is the indicator function,
to approximate the distribution $P(T_n^* \leq x)$ of $T_n^*$ in (\ref{eqn:Tn}). In practice, the mean $\mu$ in (\ref{eqn:TiBn}) is unknown and we shall replace it by the  average $\bar X_n$ of the whole sample, which turns (\ref{eqn:TiBn}) into
\begin{equation}\label{eqn:TiBn*}
T_{i,b_n} = {S_{i,i+b_n-1} - b_n \bar X_n \over \sqrt{b_n^{-1}\sum_{k=i}^{i+b_n-1}(S_{i,k} - b_n^{-1}(k-i+1) S_{i,i+b_n-1})^2}},
\end{equation}
whose empirical distribution function is given by
\begin{equation}\label{eq:F_N^*}
\widehat{F}_{n,b_n}(x) = {1 \over n-b_n+1} \sum_{i=1}^{n-b_n+1} \mathrm I(T_{i,b_n} \leq x).
\end{equation}
The asterisk in $T_{i,b_n}^*$ indicates that the centering involves the unknown population mean $\mu$, in contrast to $T_{i,b_n}$, where the centering involves instead the sample average $\bar{X}_n$. We call the above inference procedure involving using  $\widehat{F}_{n,b_n}(x)$ in (\ref{eq:F_N^*}) to approximate the distribution of $T_n^*$ in (\ref{eqn:Tn}), the {\it self-normalized block sampling (SNBS)} method.
One can then construct confidence intervals or  test hypotheses for the unknown population mean $\mu$. For instance,  to construct a one-sided $100(1-\alpha)\%$ confidence interval for  $\mu$, one gets first the $\alpha$-th quantile $q_{\alpha}$  of the empirical distribution $\widehat{F}_{n,b_n}(x)$ in (\ref{eq:F_N^*}). Since
\[
1-\alpha\approx P(T_n^*\ge q_\alpha)=P\left( \frac{S_{1,n}-n\mu}{D_n}\ge q_\alpha \right) =P\left(\mu\le \bar{X}_n-q_\alpha D_n/n\right),
\]
where $D_n$ is defined in (\ref{eq:D_n}), then the $100(1-\alpha)\%$ confidence interval is constructed as
\begin{equation}\label{eq:CI}
\left(-\infty~, ~\bar X_n -  q_{\alpha}D_n/n \right].
\end{equation}

 The idea of using block sampling to approximate distributions of self-normalized quantities is not new, and it has been applied by \citet{Fan:2010} and \citet{McElroy:Politis:2013} to long-range dependent processes with finite variances. However, the aforementioned papers did not provide a full theoretical justification for their inference procedure based on block sampling, and as commented by \citet{Fan:2010} such a task can be very nontrivial even for Gaussian processes and has been stated as an open problem. In addition, the aforementioned papers only considered the situation with finite variances, and therefore it has not been known whether one could  unify the inference procedure for processes with long-range dependence and/or heavy-tails.

Recently, \citet{jach:2012:subsampling} considered this problem in the setting of stochastic volatility models where the error term can be nicely decomposed into two independent factors, with one being a function of long-range dependent Gaussian processes while the other being i.i.d.\ heavy-tailed\footnote{As noted in Section \ref{sec:example} below, we can recover the consistency result of \citet{jach:2012:subsampling}  by replacing our normalization $D_n$ by the one found in that paper.}. But in their paper, the nonlinear function is restricted to have Hermite rank one and the choice of slowly varying functions is also greatly limited as neither $\log n$ nor $\log \log n$ are allowed. In addition, their random normalizer is specifically tailored to the aforementioned stochastic volatility model, and involves two different terms to account for the long-range dependent and heavy-tailed characteristics of the time series. Furthermore, the term in their normalizer that accounts for long-range dependence also requires the choice of an additional tuning parameter as in the estimation of the long-run variance for short-range dependent processes.
We also mention that the proof of \citet{jach:2012:subsampling}, which relies on the $\theta$-weak dependence, does not seem to be applicable in the current setting, since using our random normalizer $D_n$ in the denominator  makes the self-normalized quantity a non-Lipschitz  function of the data.

The current paper proposes to consider the use of (\ref{eq:F_N^*}) to provide a unified inference procedure without the estimation of a nuisance parameter for a wide class of processes, where the limit of the partial sum process can be a Brownian motion, an $\alpha$-stable L\'{e}vy process, a Hermite process or other processes. In Section \ref{sec:theory}, we develop an asymptotic theory for the self-normalized block sums and establish the theoretical consistency of the aforementioned method, namely,
\begin{equation}\label{eq:consistency}
|\widehat{F}_{n,b_n}(x) - P(T_n^* \leq x)| \rightarrow 0
\end{equation}
in probability as $n\rightarrow\infty$.

\section{Asymptotic Theory}\label{sec:theory}
We establish the asymptotic consistency of self-normalized block sampling for the following two classes of stationary processes:  (a) nonlinear transforms of  Gaussian stationary processes (called Gaussian subordination), and (b) those satisfying strong mixing conditions.
The first allows for long-range dependence and non-central limits, while the second  involves  short-range dependent processes.  Both classes allow  for heavy-tails with infinite variance.

Let $D[0,1]$ be the   space of c\`adl\`ag (right continuous with left limits) functions defined on $[0,1]$, endowed with Skorokhod's $M_2$ topology. The $M_2$ topology is weaker than the other topologies proposed by \citet{skorokhod:1956:limit}, in particular, weaker than the most commonly used $J_1$ topology. A sequence of function $x_n(t)\in D[0,1]$ converges to $x(t)\in D[0,1]$ in $M_2$ topology as $n\rightarrow\infty$, if and only if $\lim_{n}\sup_{t_1\le t\le t_2} x_n(t)=\sup_{t_1\le t\le t_2} x(t)$
 and $\lim_{n}\inf_{t_1\le t\le t_2} x_n(t)=\inf_{t_1\le t\le t_2} x(t)$ for any $t_1, t_2$ at continuity points of $x(t)$ (see statement 2.2.10 of \citet{skorokhod:1956:limit}).

We consider the $M_2$ topology instead of $J_1$ since there are known examples in the heavy tailed case  where convergence  fails under $J_1$ but holds under $M_2$ (see \citet{avram:taqqu:1992:weak}, \citet{tyran:2010:functional} and \citet{basrak:2012:functional}).
To apply the continuous mapping argument, we need the following lemma.
\begin{lemma}\label{Lem:int cont}
Integration on $[0,1]$ is  a continuous functional for $D[0,1]$ under the $M_2$ topology.
\end{lemma}
\begin{proof}
Suppose that $x_n(t)\rightarrow x(t)$ in the $M_2$ topology.
For any partition  $\mathcal{T}=\{0=t_0< t_1<\ldots<t_{k-1}< t_k=1\}$, define  $m_{i,n}=\inf_{t_{i-1}\le t\le t_{i}} x_n(t)$, $M_{i,n}=\sup_{t_{i-1}\le t\le t_i} x_n(t)$, $m_i=\inf_{t_{i-1}\le t\le t_{i}} x(t)$ and $M_i=\sup_{t_{i-1}\le t\le t_i} x(t)$, $i=1,\ldots,k$. Note that
\begin{align}
\sum_{i=1}^k m_{i,n}(t_i-t_{i-1})\le &\int_0^1 x_n(t)dt \le \sum_{i=1}^k M_{i,n}(t_i-t_{i-1}),\notag\\
\sum_{i=1}^k m_i(t_i-t_{i-1})\le &\int_0^1 x(t)dt \le \sum_{i=1}^kM_i(t_i-t_{i-1}).\label{eq:Riemann 1}
\end{align}
The function $x(t)$ is Riemann integrable since, as an element in $D[0,1]$, it is a.e.\ continuous and bounded on $[0,1]$. Riemann integrability implies that for any $\epsilon>0$, one can choose a partition $\mathcal{T}$ so that
\begin{equation} \label{eq:Riemann 2}
0\le \sum_{i=1}^kM_i(t_i-t_{i-1})- \sum_{i=1}^k m_i(t_i-t_{i-1})< \epsilon.
\end{equation}
Modify the partition, if necessary, so that all the  $t_i$'s are  at continuity points of $x(t)$, without changing (\ref{eq:Riemann 2}). This is possible since $x(t)$ has at most countable discontinuity points and is bounded.
By the characterization of convergence in $D[0,1]$ with $M_2$ topology, we have
\begin{align}\label{eq:Riemann 3}
\lim_n\sum_{i=1}^k m_{i,n}(t_i-t_{i-1})=\sum_{i=1}^k m_{i}(t_i-t_{i-1}),\notag\\
\lim_n\sum_{i=1}^k M_{i,n}(t_i-t_{i-1})=\sum_{i=1}^k M_{i}(t_i-t_{i-1}).
\end{align}
Combining (\ref{eq:Riemann 1}), (\ref{eq:Riemann 2}) and (\ref{eq:Riemann 3}) concludes that $\limsup_n|\int_0^1 x_n(t)dt- \int_0^1 x(t)dt|\le \epsilon$.
\end{proof}

\subsection{Results in the Gaussian subordination case}
Let
\begin{equation}\label{eq:Z vec}
\{\mathbf{Z}_i=(Z_{i,1},\ldots, Z_{i,J}), ~ i\in \mathbb{Z}\}
\end{equation}
be an $\mathbb{R}^J$-valued Gaussian stationary process satisfying $\E Z_{i,j}=0$ for any $i,j$. Define
\begin{equation}\label{e:VZ}
\mathbf{Z}_p^q= \left(\mathbf{Z}_p,\ldots, \mathbf{Z}_q \right).
\end{equation}
We shall view $\mathbf{Z}_p^q$ as a vector of dimension $J\times (q-p+1)$ involving observations from time $p$ to time $q$.
The covariance matrix of $\mathbf{Z}_1^m$ will be written  for convenience  as a four-dimensional array involving $i_1,i_2,j_2,j_2$:
\begin{equation}\label{eq:Sigma_m}
\Sigma_m=\Big(\gamma_{j_1,j_2}(i_2-i_1):= \E Z_{i_1,j_1}Z_{i_2,j_2}\Big)_{1\le i_1,i_2\le m, 1\le j_1,j_2\le J} .
\end{equation}
We assume throughout that $\Sigma_m$ is non-singular for every $m\in \mathbb{Z}_+$.
The cross-block covariance matrix between $\mathbf{Z}_{1}^m$ and $\mathbf{Z}_{k+1}^{k+m}$ is
\begin{equation}\label{eq:Sigma_k,m}
\Sigma_{k,m}=\Big(\gamma_{j_1,j_2}(i_2+k-i_1):= \E Z_{i_1,j_1}Z_{i_2+k,j_2}\Big)_{1\le i_1,i_2\le m, 1\le j_1,j_2\le J}.
\end{equation}
Let
 $\rho(\cdot,\cdot)$ denote the canonical correlation (maximum correlation coefficient) between $L^2(\Omega)$ random vectors $\mathbf{U}=(U_1,\ldots,U_p)$ and $\mathbf{V}=(V_1,\ldots,V_q)$. Let $\langle\cdot,\cdot \rangle$ denote the inner product in an Euclidean space of a suitable dimension. Then
\begin{equation}\label{eq:rho}
\rho(\mathbf{U},\mathbf{V})=\sup_{\mathbf{x}\in \mathbb{R}^p, \mathbf{y}\in \mathbb{R}^q} \left|\Corr\Big(\langle\mathbf{x},\mathbf{U}\rangle, \langle\mathbf{y}, \mathbf{V}\rangle\Big)\right|.
\end{equation}
Let $\rho_{k,m}$ be the between-block canonical correlation:
\begin{equation}\label{eq:spec radius}
\rho_{k,m}=\rho\left(\mathbf{Z}_1^m, \mathbf{Z}_{k+1}^{k+m}\right).
\end{equation}

We now introduce the  assumptions for the self-normalized block sampling procedure.  $\{X_i\}$ is  the stationary process (time series) we observe.
\bigskip

\fbox{\parbox{\textwidth}{
\begin{enumerate}[\textbf{A}1.]
\item  $X_i=G(\mathbf{Z}_i,\ldots,\mathbf{Z}_{i-l})=G(\mathbf{Z}_{i-l}^{i})$ with mean $\mu=\E X_i$, where $\{\mathbf{Z}_i\}$ is a vector-valued stationary Gaussian process as in (\ref{eq:Z vec}), and $l$ is a fixed non-negative integer.
\item We have  weak convergence in $D[0,1]$ endowed with the $M_2$ topology for the partial sum:
\[
\left\{\frac{1}{n^{H}\ell(n)}(S_{\lfloor nt \rfloor} - n\mu),\ 0 \leq t \leq 1\right\} \Rightarrow \left\{Y(t),\ 0 \leq t \leq 1\right\},
\]
for some nonzero $H$-sssi  process $Y(t)$, where $0<H<1$ and $\ell(\cdot)$ is a slowly varying function.
\item As $n\rightarrow\infty$, the block size  $b_n\rightarrow\infty$, $b_n=o(n)$, and satisfies
\begin{equation}\label{eq:block cond gen}
\sum_{k=0}^{n} \rho_{k,l+b_n}=o(n),
\end{equation}
where $\rho_{k,m}$ is the between-block canonical correlation defined in (\ref{eq:spec radius}).
\end{enumerate}
}}

\bigskip
\begin{remark}
The data-generating specification in A1 allows us to
get a variety of limits in A2, covering short-range dependence, long-range dependence,  and heavy tails. When the covariance function of $X(n)$ is absolutely summable (short-range dependence), one typically gets in A2 convergence to Brownian motion (see, e.g., \citet{breuer:major:1983:central}, \citet{ho:sun:1987:central} and \citet{chambers:slud:1989:central}).
When the covariance of $X(n)$ is regularly varying of  order between $-1$ and $0$ (long-range dependence), one may get in A2 convergence to  the Hermite-type  processes  (see, e.g., \citet{Taqqu:1975},  \citet{Dobrushin:Major:1979}, \citet{taqqu:1979:convergence} and \citet{arcones:1994:limit}).

Moreover, as shown in \citet{sly:heyde:2008:nonstandard}  in the case $J=1$,  when $G(\cdot)$ is chosen such that $X(n)$ is short-range dependent and  heavy-tailed, so that $X(n)$ has infinite variance but finite mean,  one can obtain in  A2, convergence to an infinite-variance $\alpha$-stable L\'evy process; if $X(n)$ is long-range dependent and heavy-tailed, then the limit may be a finite-variance Hermite process, even though $X(n)$ may have infinite variance. All these situations are allowed under Assumptions A1--A3.

For sufficient conditions for Assumption A3 to hold, see  Proposition \ref{Pro:new bound} and Section \ref{sec:ass A3}.
\end{remark}
Since the denominators in (\ref{eqn:general self norm limit}) are nonzero almost surely,
Assumption A2, Lemma \ref{Lem:int cont} and the Continuous Mapping Theorem imply the following (see \citet{kallenberg:2006:foundations}, Corollary 4.5):
\begin{lemma}\label{Lem:cont mapping}
$T_{i,b_n}^*$ in (\ref{eqn:TiBn}) converges in distribution to $T$ in (\ref{eqn:general self norm limit}).
\end{lemma}

The following result allows us to relate the correlation of nonlinear functions to the correlation of linear functions. 
\begin{lemma}\label{Lem:canonical}
Let $(\mathbf{Z_i})_{i\in \mathbb{Z}}$ be a centered $\mathbb{R}^J$-valued Gaussian stationary process as in (\ref{eq:Z vec}),  and let $\mathbf{Z}_p^q$ be defined as in (\ref{e:VZ}).  Let $\mathcal{F}_{Jm}$ be the set of all functions $F$ on $\mathbb{R}^{Jm}$ satisfying $\E F(\mathbf{Z}_1^m)^2<\infty$.  Then for $k\ge m$, one has
\begin{equation}\label{eq:maximal corr eq}
\sup_{F,G\in \mathcal{F}_{Jm}}\left|\Corr\big(F(\mathbf{Z}_1^m),G(\mathbf{Z}_{k+1}^{k+m})\big)\right|=  \rho\left(\mathbf{Z}_1^m, \mathbf{Z}_{k+1}^{k+m}\right)=\rho_{k,m}.
\end{equation}
\end{lemma}
\begin{proof}
The  equality is the well-known Gaussian maximal correlation equality. See, e.g., Theorem 1  of \citet{kolmogorov:razanov:1960:strong} or Theorem 10.11 of \citet{janson:1997:gaussian}.
\end{proof}


Our goal is to show  that (\ref{eq:consistency}) holds, namely, $\widehat{F}_{n,b_n}$ is a consistent estimator of $P(T_n^*\le x)$. This will be a consequence of the following theorem.
\begin{theorem}\label{Thm:main gaussian}
Assume that Assumptions A1--A3 hold. Let $F(x)$ be the CDF (cumulative distribution function) of  $T$ in (\ref{eqn:general self norm limit}), and let $\widehat{F}_{n,b_n}(x)$ be as in (\ref{eq:F_N^*}).
 As $n\rightarrow\infty$,  we have
\begin{equation}\label{eq:F conv mixing}
\widehat{F}_{n,b_n}(x)\ConvP F(x), \quad  x\in C(F),
\end{equation}
where  $C(F)$ denotes the set of continuity points of $F(x)$.
If  $F(x)$ is continuous, then (\ref{eq:F conv mixing}) can be strengthened to
\begin{equation}\label{eq:F conv uniform mixing}
\sup_x \left|\widehat{F}_{n,b_n}(x) - F(x) \right|\rightarrow 0 \quad \text{ in probability}.
\end{equation}
\end{theorem}
\begin{proof}
~\\
\noindent\emph{Step 1.}
Let $\widehat{F}_{n,b_n}^*(x)$ be as in (\ref{eq:F_N}). To prove (\ref{eq:F conv mixing}), we  first show that
\begin{equation}\label{eq:reduced goal}
\widehat{F}_{n,b_n}^*(x)\ConvP F(x),\quad x\in C(F),
\end{equation}
where we have replaced $\widehat{F}_{n,b_n}(x)$ by $\widehat{F}_{n,b_n}^*(x)$.
A bias-variance decomposition yields:
\begin{align*}
\E \left( \Big[\widehat{F}_{n,b_n}^*(x) - F(x)\Big]^2\right)
&= [\E \widehat{F}_{n,b_n}^*(x)]^2 -\E[ 2  F(x) \widehat{F}_{n,b_n}^*(x) ]+ F(x)^2 + \E [\widehat{F}_{n,b_n}^*(x)^2 ]-[\E\widehat{F}_{n,b_n}^*(x)]^2
\\&= 
 \Big[\E\widehat{F}_{n,b_n}^*(x) - F(x)\Big]^2+ \Big[\E[\widehat{F}_{n,b_n}^*(x)^2]-[\E\widehat{F}_{n,b_n}^*(x)]^2\Big]\\
 &
=\big[P(T_{i,b_n}^*\le x)-P(T\le x)\big]^2+
\Var\big[\widehat{F}_{n,b_n}^*(x)\big].
\end{align*}
By Lemma \ref{Lem:cont mapping}, the squared bias $[P(T_{i,b_n}^*\le x)-P(T\le x)]^2$ converges to zero for $x\in C(F)$ as $b_n\rightarrow\infty$.
We thus need to show that $\Var[\widehat{F}_{n,b_n}^*(x)]\rightarrow 0$. By the stationarity of $\{X_i\}$, which implies the stationarity of $\{T_{i,b_n}^*\}$ viewed as a process indexed by $i$, one has
\begin{align}\label{eq:var bound}
\Var[\widehat{F}_{n,b_n}^*(x)]&=\Var\left[ \frac{1}{n-b_n+1}\sum_{i=1}^{n-b_n+1} \I\{T_{i,b_n}^*\le x\} \right]=\frac{1}{(n-b_n+1)^2}\sum_{i,j=1}^{n-b_n+1} \Cov\left[\I\{T_{i,b_n}^*\le x\},\I\{T_{j,b_n}^*\le x\} \right]\notag
 \\&
 \le \frac{2}{n-b_n+1} \sum_{k=0}^{n} \left|\Cov\left[\I\{T_{1,b_n}^*\le x\},\I\{T_{k+1,b_n}^*\le x\}\right]\right|,
\end{align}
since for any covariance function $\gamma(\cdot)$ of a stationary sequence, we have
\[
\sum_{i,j=1}^p |\gamma(i-j)|\le \sum_{|k|<p} (p-|k|) |\gamma(k)| \le 2p\sum_{k=0}^p |\gamma(k)|.
\]
In view of Assumption A1, $X_i$ depends on $\mathbf{Z}_{i},\ldots,\mathbf{Z}_{i-l}$.  By (\ref{eqn:TiBn}),   $T_{i,b_n}^*$  is a function of  $X_i,\ldots,X_{i+b_n-1}$. Hence $T_{1,b_n}^*$ depends not only on $\mathbf{Z}_1,\ldots,\mathbf{Z}_{b_n}$, but also on $\mathbf{Z}_{1-l},\ldots,\mathbf{Z}_{0}$, and $T_{k+1,b_n}^*$ depends on $\mathbf{Z}_{k+1-l},\ldots,\mathbf{Z}_{k+b_n}$.
We shall now apply Lemma \ref{Lem:canonical} with the same $k$ and $m=l+b_n$.
Then
when $k\ge l+ b_n$, one has
\begin{align}\label{eq:cov bound}
\left|\Cov[\I\{T_{1,b_n}^*\le x\},\I\{T_{k+1,b_n}^*\le x\}]\right|\le \frac{1}{4}\left|\Corr[\I\{T_{1,b_n}^*\le x\},\I\{T_{k+1,b_n}^*\le x\}]\right| \le  \frac{1}{4} \rho_{k,b_n+l}
\end{align}
where we have used the following fact\footnote{If $0\le X\le 1$, then $\mu =\E X \in [0,1]$,  $\E X^2\le \mu$ and $\Var[X] \le \mu-\mu^2$ is maximized at $\mu=1/2$, so that $\Var[X]\le 1/4$ (for more general results, see \citet{dharmadhikari:1989:upper}, Lemma 2.2).}:  if  $0\le X\le 1$,  then $\Var[X]\le 1/4$. We have
\begin{align}\label{eq:var bound gaussian}
\Var[\widehat{F}_{n,b_n}^*(x)]&\le \frac{1}{2(n-b_n+1)} \sum_{k=0}^n\rho_{k,b_n+l},
\end{align}
which converges to zero
because of Assumption A3. Hence $\widehat{F}_{n,b_n}^*(x)\ConvP F(x)$ for  $x\in C(F)$. Step 1 of the proof is now complete.

\medskip
\noindent\emph{Step 2.}
We now show that
\[
\widehat{F}_{n,b_n}(x)\ConvP F(x)\qquad \text{for  $x\in C(F)$,}
\]
  that is, we go from (\ref{eq:reduced goal}) to (\ref{eq:F conv mixing}). To do so,  we follow the proof of Theorem 11.3.1 of \citet{politis:1999:subsampling}, and express (\ref{eq:F_N^*}) as
\begin{equation}\label{eq:F_n rewrite}
\widehat{F}_{n,b_n}(x)=\frac{1}{n-b_n+1} \sum_{i=1}^{n-b_n+1} \I\{T_{i,b_n}^*\le x+ b_n(\bar{X}_n-\mu)/D_{i,b_n}\},
\end{equation}
where $D_{i,b_n}$ is as in (\ref{eqn:TiBn}). The goal is to show that $b_n(\bar{X}_n-\mu)/D_{i,b_n}$ is negligible. For $\epsilon>0$, define
\begin{align}\label{eq:R_n def}
R_n(\epsilon)
&=\frac{1}{n-b_n+1} \sum_{i=1}^{n-b_n+1} \I\{ b_n(\bar{X}_n-\mu)/D_{i,b_n}\le
\epsilon\}
\\&=\frac{1}{n-b_n+1} \sum_{i=1}^{n-b_n+1}  \I\{ (b_n^{H} \ell(b_n))^{-1}   D_{i,b_n}  \ge \epsilon^{-1} b_n(\bar{X}_n-\mu)(b_n^{H} \ell(b_n))^{-1}\}.\notag
\end{align}
Since $R_n(\epsilon)$ is an average of indicators, we have $R_n(\epsilon)\le 1$. Our goal is to show that $R_n(\epsilon)\ConvP 1$. Note that as $n\rightarrow\infty $,
$$
\frac
{D_{i,b_n}}
{b_n^{H} \ell(b_n)}
 =
\frac
{1}
{b_n^{H} \ell(b_n)}
{\left(b_n^{-1}\sum_{k=i}^{i+b_n-1}\Big(S_{i,k} - b_n^{-1}(k-i-1) S_{i,i+b_n-1}\Big)^2\right)^{1/2}}
$$
converges in distribution to $D$ in (\ref{eq:D}) by  Assumption A2 and continuous mapping. Moreover,   since $b_n=o(n)$, $H<1$  and $n(\bar{X}_n-\mu)n^{-H}\ell(n)^{-1} $  converges in distribution to $Y(1)$ by  Assumption A2, we have
\[
b_n(\bar{X}_n-\mu)(b_n^{H} \ell(b_n))^{-1}=n(\bar{X}_n-\mu)n^{-H}\ell(n)^{-1} ~
\frac{n^{H-1}\ell(n)}{b_n^{H-1} \ell(b_n)} \ConvP 0.
\]
 Hence for any $\delta>0$, with probability tending to $1$ as $n\rightarrow \infty$, one has
\begin{equation}\label{eq:R_n>=}
1\ge R_n(\epsilon)\ge  \frac{1}{n-b_n+1} \sum_{i=1}^{n-b_n+1}\I\{ (b_n^{H} \ell(b_n))^{-1} D_{i,b_n}    \ge \delta \epsilon^{-1} \}.
\end{equation}
Since as  $T_{i,b_n}^*$ in Step 1, $D_{i,b_n}$ is  also a function of  $X_i,\ldots,X_{i+b_n-1}$, we can follow a same argument as in Step 1, replacing $T_{i,b_n}^*$ by $(b_n^{H} \ell(b_n))^{-1} D_{i,b_n}$ to obtain a similar result as in (\ref{eq:reduced goal}), namely that the empirical distribution of $(b_n^{H} \ell(b_n))^{-1} D_{i,b_n}$ converges in probability to that of $D$ at all points of continuity of the distribution of $D$. Therefore
\begin{equation}\label{eq:again apply}
\frac{1}{n-b_n+1} \sum_{i=1}^{n-b_n+1}\I\{ (b_n^{H} \ell(b_n))^{-1} D_{i,b_n}    \ge \delta \epsilon^{-1} \} \ConvP P(D\ge \delta\epsilon^{-1})
\end{equation}
for $\delta\epsilon^{-1}$  at continuity point of the CDF of $D$.
Since $P(D>0)=1$, we can choose $\delta$ small enough to make $P(D\ge \delta\epsilon^{-1})$ as close to $1$ as desired. In view of (\ref{eq:R_n>=}) and (\ref{eq:again apply}), we conclude that as $n\rightarrow\infty$,
\begin{equation}\label{eq:R_n conv 1}
R_n(\epsilon)\ConvP 1
\end{equation}
for any $\epsilon>0$.
Now notice that each summand in the sum  (\ref{eq:F_n rewrite}) satisfies
\begin{eqnarray}\label{eq:to be modified}
 & &\I\{T_{i,b_n}^*\le x+ b_n(\bar{X}_n-\mu)/D_{i,b_n}\}\notag\\
 & = &\Big[\I\{T_{i,b_n}^*\le x+ b_n(\bar{X}_n-\mu)/D_{i,b_n}\}\Big]\Big[\I\{b_n(\bar{X}_n-\mu)/D_{i,b_n}\le \epsilon\}+\I\{b_n(\bar{X}_n-\mu)/D_{i,b_n}>\epsilon\}\Big]
\notag\\
 & \le &\I\{T_{i,b_n}^*\le x+ \epsilon\}+ \I\{b_n(\bar{X}_n-\mu)/D_{i,b_n}>\epsilon\},
\end{eqnarray}
so that by plugging these inequalities in (\ref{eq:F_n rewrite}) and using (\ref{eq:R_n def}),
we get
\[
\widehat{F}_{n,b_n}(x)\le \widehat{F}_{n,b_n}^*(x+\epsilon)+  1-R_n(\epsilon).
\]
But by (\ref{eq:R_n conv 1}), $R_n(\epsilon)\ConvP 1$. So for any $\gamma>0$, one has
$$
\widehat{F}_{n,b_n}(x)\le \widehat{F}_{n,b_n}^*(x+\epsilon)+\gamma
$$
with probability tending to $1$ as $n\rightarrow\infty$.
We can now  use  (\ref{eq:reduced goal}) to replace $\widehat{F}_{n,b_n}^*(x+\epsilon)$ by $F(x+\epsilon)$,
so that for arbitrary $\gamma'>\gamma$, and for any  $x+\epsilon\in C(F)$, one has  $\widehat{F}_{n,b_n}(x)\le F(x+\epsilon)+\gamma'$  with probability tending to $1$ as $n\rightarrow \infty$. Now  letting $\epsilon\downarrow 0$ through $x+\epsilon\in C(F)$ and using the continuity of $F(\cdot)$ at $x$, one gets with probability tending to $1$ that
\begin{equation}\label{eq:upper}
\widehat{F}_{n,b_n}(x)\le F(x)+\gamma'',\quad x\in C(F),
\end{equation}
for any $\gamma''>\gamma'$.

A similar argument, which replaces (\ref{eq:to be modified}) by
\begin{align*}
\I\{T_{i,b_n}\le x\} \geq \I\{T_{i,b_n}^*\le x- \epsilon\} - \I\{b_n(\bar{X}_n-\mu)/D_{i,b_n}<-\epsilon\},
\end{align*}
will show that for any $\gamma''>0$, with probability tending to 1,
\begin{equation}\label{eq:lower}
\widehat{F}_{n,b_n}(x)\ge F(x)-\gamma'',\quad x\in C(F).
\end{equation}
Combining (\ref{eq:upper}) and (\ref{eq:lower}), one gets
\[
P(|\widehat{F}_{n,b_n}(x)-F(x)|\le \gamma'')\rightarrow 1
\]
as $n\rightarrow\infty$, and thus
(\ref{eq:F conv mixing}) holds.

\medskip

\noindent\emph{Step 3.}
We now show (\ref{eq:F conv uniform mixing}). If $F(x)$ is continuous, then by the already established (\ref{eq:F conv mixing}), we have $\widehat{F}_{n,b_n}(x)\rightarrow F(x)$ in probability for any $x\in \mathbb{R}$. Let $n_i$ be an arbitrary subsequence, one can then choose a further subsequence of $n_i$, still denoted as $n_i$, so that $\widehat{F}_{n_i}(x)\rightarrow F(x)$ almost surely for all rational $x$ by a diagonal subsequence argument. Then by Lemma A9.2 (ii) of \citet{gut:2006:probability}, $\sup_{x\in \mathbb{R}}|\widehat{F}_{n_i}(x)-F(x)|\rightarrow 0$ almost surely, and therefore $\sup_{x\in \mathbb{R}}|\widehat{F}_{n}(x)-F(x)|\rightarrow 0$  in probability. Hence (\ref{eq:F conv uniform mixing}) is proved.
\end{proof}

Consistency  (\ref{eq:consistency}) is a simple corollary of Theorem \ref{Thm:main gaussian}.
\begin{corollary}\label{Cor:gaussian}
Assume that  Assumptions A1--A3 hold. Then  as $n\rightarrow\infty$,
\begin{equation}\label{eq:consistency F}
|\widehat{F}_{n,b_n}(x) - P(T_n^* \leq x)|\rightarrow 0   \quad \text{ in probability. }
\end{equation}
for $x\in C(F)$. If $F(x)$ is continuous, then the preceding convergence can be strengthened to
\begin{equation}\label{eq:uniform consistency}
\sup_{x\in \mathbb{R}}|\widehat{F}_{n,b_n}(x) - P(T_n^* \leq x)| \rightarrow 0 \quad \text{ in probability. }
\end{equation}
\end{corollary}
\begin{proof}
The first result (\ref{eq:consistency F}) follows directly from the triangle inequality
\[
|\widehat{F}_{n,b_n}(x) - P(T_n^* \leq x)|\le |\widehat{F}_{n,b_n}(x) - F(x)| + |P(T_n^* \leq x)-F(x)|,
\]
where $x\in C(F)$ and $F(x)=P(T\le x)$, by combining Theorem \ref{Thm:main mixing} or \ref{Thm:main gaussian} with (\ref{eqn:general self norm limit}).  For the second result (\ref{eq:uniform consistency}), one  uses also the fact that (\ref{eqn:general self norm limit}) implies  $\sup_{x\in \mathbb{R}}|P(T_n^* \leq x)-F(x)|\rightarrow 0$ as $n\rightarrow \infty$ if $F(x)$ is continuous (see again Lemma A9.2 (ii) of \citet{gut:2006:probability}).
\end{proof}

\citet{bai:taqqu:2015:canonical} recently proved the following proposition, showing  that the bound (\ref{eq:block cond gen}) holds for a large class of models with long-range dependence. Thus, for these models, one has the freedom to choose any $b_n=o(n)$, irrespective of the long-range dependence parameter $H$.
\begin{proposition}[\citet{bai:taqqu:2015:canonical}, Theorem 2.2 and 2.3]\label{Pro:new bound}
Consider the case $J=1$. Suppose that the spectral density of the underlying Gaussian $\{Z_i\}$ is given by
\[
f(\lambda)=f_H(\lambda)f_0(\lambda),
\]
where $f_H(\lambda)=|1-e^{i\lambda}|^{-2H+1}$, $1/2<H<1$, and  $f_0(\lambda)$ is a spectral density  which corresponds to a covariance function (or  Fourier coefficient)
$\gamma_0(n)=\int_{-\pi}^{\pi} f_0(\lambda) e^{in\lambda}d\lambda$. Assume that the following hold:
\begin{enumerate}[(a)]
\item There exists $c_0>0$ such that $f_0(\lambda)\ge c_0$ for all $\lambda\in (-\pi,\pi]$;
\item $\sum_{n=-\infty}^\infty|\gamma_0(n)|<\infty$;
\item $\gamma_0(n)=o(n^{-1})$.
\end{enumerate}
Then the condition (\ref{eq:block cond gen}) in Assumption A3 holds if $b_n=o(n)$. The result extends to the case where the underlying Gaussian $\{\mathbf{Z}_i\}$ is $J$-dimensional with independent components.
\end{proposition}
In Proposition \ref{Pro:new bound}, $f_H(\lambda)$ is the spectral density of a FARIMA($0,d,0$) sequence with $d=H-1/2$, and $f_0(\lambda)$ is the spectral density of a sequence with short-range dependence.

Under the assumptions in Proposition \ref{Pro:new bound}, the spectral density $f(\lambda)$  cannot have a slowly varying factor which diverges to infinity or converges to zero at $\lambda=0$, because $f_0(\lambda)$  is bounded away from infinity and zero. For $H\in (1/2,1)$, the $\mathrm{FARIMA}(p,d,q)$ model with $d=H-1/2$ and the fractional Gaussian noise model satisfy the assumptions of Proposition \ref{Pro:new bound}.  See Examples 2.1 and 2.2 of \citet{bai:taqqu:2015:canonical}.

We thus have the following result which we formulate for simplicity in the univariate case $J=1$.
\begin{corollary}
Assume that Assumptions A1-A2  hold with $J=1$, and the underlying Gaussian $\{Z_i\}$ satisfies the assumptions in Proposition \ref{Pro:new bound}. If $b_n\rightarrow\infty$ and $b_n=o(n)$, then the conclusions of Theorem \ref{Thm:main gaussian} and Corollary \ref{Cor:gaussian} hold.

\end{corollary}

\subsection{Further analysis of Assumption A3}\label{sec:ass A3}
In this section, we discuss the critical  Assumption A3, which involves the covariance structure of the underlying Gaussian $\{\mathbf{Z}_i\}$.  In particular, we shall give the general  bound (\ref{eq:bound lambda_k,m}) below for the canonical correlation $\rho_{k,m}$ in (\ref{eq:spec radius}), and discuss how it relates to Assumption A3. As noted in Proposition \ref{Pro:new bound}, however, this bound, in the long memory case, can be improved substantially so as  to provide more flexibility   on the choice of the block size $b_n$.

To state this general bound, define
\begin{equation}\label{e:Mmax}
M_\gamma(k)=\max_{n> k} \max_{1\le j_1,j_2\le J} |\gamma_{j_1,j_2}(n)|,
\end{equation}
and
\begin{equation}\label{eq:lambda min}
\text{$\lambda_m=$ the {\it minimum} eigenvalue of  $\Sigma_m$}.
\end{equation}
Note that $\lambda_m>0$ since $\Sigma_m$ is assumed to be positive definite.
\begin{lemma}\label{Lem:naive bound}
Let $\rho_{k,m}$ be as in (\ref{eq:spec radius}), $M_\gamma(k)$ be as in (\ref{e:Mmax}) and $\lambda_m$ be as  in (\ref{eq:lambda min}). We have the bound
\begin{equation}\label{eq:bound lambda_k,m}
\rho_{k,m}\le \min \left\{Jm \frac{M(k-m)}{\lambda_m}, 1\right\}.
\end{equation}
\end{lemma}
\begin{proof}
Let $\mathbf{x}$ and $\mathbf{y}$ be (column) vectors in $\mathbb{R}^{Jm}$.
Note that each $\mathbf{Z}_1^m=(\mathbf{Z}_1,\cdots,\mathbf{Z}_m)$ and $\mathbf{Z}_{k+1}^{k+m}=(\mathbf{Z}_{k+1},\cdots,\mathbf{Z}_{k+m})$ are $Jm$-dimensional Gaussian vectors translated by $k$ units in the time index. Therefore by (\ref{eq:rho}),
\begin{equation}\label{eq:rho alternative}
\rho_{k,m}=\rho\left(\mathbf{Z}_1^m, \mathbf{Z}_{k+1}^{k+m}\right)=\sup_{\mathbf{x}, \mathbf{y}\in \mathbb{R}^{Jm}}
\frac{\E\left[ \langle\mathbf{x}, \mathbf{Z}_1^m \rangle \langle \mathbf{y}, \mathbf{Z}_{k+1}^{k+m}\rangle \right]}{\Big(\Var[ \langle\mathbf{x}, \mathbf{Z}_1^m\rangle]\Big)^{1/2} \Big(\Var[\langle\mathbf{y}, \mathbf{Z}_{k+1}^{k+m}\rangle]\Big)^{1/2} }= \sup_{\mathbf{x}, \mathbf{y}\in \mathbb{R}^{Jm}} \frac{\mathbf{x}^T \Sigma_{k,m} \mathbf{y}}{\sqrt{\mathbf{x}^T \Sigma_m \mathbf{x}} \sqrt{\mathbf{y}^T \Sigma_m \mathbf{y}}},
\end{equation}
where $\Sigma_m$ is as in (\ref{eq:Sigma_m}),
$\Sigma_{k,m}$ is as in (\ref{eq:Sigma_k,m}).
By relations 6.58(a) and 6.62(a) in \citet{seber:2008:matrix}, one has
\begin{equation}\label{eq:opt lambda_k,m}
\rho_{k,m}= \sup_{\mathbf{x}, \mathbf{y}\in \mathbb{R}^{Jm}}\frac{\left|\mathbf{x}^T \Sigma_{k,m} \mathbf{y}\right|}{\sqrt{\mathbf{x}^T \Sigma_m \mathbf{x}} \sqrt{\mathbf{y}^T \Sigma_m \mathbf{y}}} \le  \sup_{\mathbf{x}, \mathbf{y}\in \mathbb{R}^{Jm}}\frac{1}{\lambda_m}\frac{\left|\mathbf{x}^T \Sigma_{k,m} \mathbf{y}\right|}{\|\mathbf{x}\|\|\mathbf{y}\|}\le \frac{1}{\lambda_m} \sigma_{k,m},
\end{equation}
where $\lambda_m$ is the smallest eigenvalue of $\Sigma_m$, and  $\sigma_{k,m}$ is the maximum  singular value\footnote{Note that $\Sigma_{k,m}$ is not a symmetric matrix. The square of its singular values are the eigenvalues of $\Sigma_{k,m}^T \Sigma_{k,m}$, which is symmetric and non-negative definite. 
} of $\Sigma_{k,m}$.
By \citet{seber:2008:matrix} 4.66(b) and 4.67(b),  $\sigma_{k,m}$ is bounded by the linear size of the matrix $\Sigma_{k,m}$ times the maximum absolute value of all the elements of the matrix.
Since the matrix $\Sigma_{k,m}$ has linear size $Jm$, we have
\[
\sigma_{k,m}\le Jm \max_{1\le i_1,i_2\le m}\max_{1\le j_1,j_2\le J}|\gamma_{j_1,j_2}(i_2+k-i_1)| \le Jm  \max_{n> k-m} \max_{1\le j_1,j_2\le J} |\gamma_{j_1,j_2}(n)|
=JmM_\gamma(k-m).
\]
The bound (\ref{eq:bound lambda_k,m}) is then obtained by noting that $\rho_{k,m}\le 1$ in view of (\ref{eq:maximal corr eq}).
\end{proof}

\begin{example}
Consider the important scalar case $J=1$,
where $\mathbf{Z}_i=Z_i$. Denote the covariance function of $\{Z_i\}$ by $\gamma(n)$ and  its spectral density by  $f(\omega)$.
 In this case, it is known that
$\Sigma_m$ is non-singular for any $m$ if $\lim_{n\rightarrow\infty}\gamma(n)=0$ (see Proposition 5.1.1 of \citet{brockwell:1991:time}), and that the minimum eigenvalue $\lambda_m$ satisfies
\begin{equation}\label{eq:lambda lower bound}
\lambda_m\ge 2\pi~ \mathrm{ess\,inf}_\omega f(\omega),\quad \text{and}\quad \lim_{m\rightarrow\infty} \lambda_m =2\pi~\mathrm{ess\,inf}_\omega f(\omega),
\end{equation}
where ``$\mathrm{ess\,inf}$'' denotes the essential infimum with respect to Lebesgue measure on $[-\pi,\pi)$  (see \citet{grenander:szego:1958:toeplitz},  Chapter 5.2).
If $J=1$, $M_\gamma(k)$ also reduces to
\begin{equation}\label{eq:M_gamma one dim}
M_\gamma(k)=\max_{n> k}|\gamma(n)|.
\end{equation}
\end{example}

\begin{remark}\label{Rem:indep}
Consider the vector case but suppose that $\{Z_{i,1}\},\ldots,\ldots, \{Z_{i,J}\}$ are mutually independent, i.e.,
$
\gamma_{j_1,j_2}(n)=\gamma_{j_1,j_2}(n)\I\{j_1=j_2\}.
$
Let $\Gamma_{m,j}=\left(\gamma_{j,j}(i_1-i_2)\right)_{1\le i_1,i_2\le m}$.
In this case, we have a  block-diagonal $\Sigma_m=\mathrm{diag}(\Gamma_{m,1},\ldots,\Gamma_{m,J})$.
Let  $\Gamma_{k,m,j}=\left(\gamma_{j,j}(i_2+k-i_1)\right)_{1\le i_1,i_2\le m}$.
 We also have a block-diagonal $\Sigma_{k,m}=\mathrm{diag}(\Gamma_{k,m,1},\ldots,\Gamma_{k,m,J})$. Let $\rho_{k,m,j}$ be the between-block canonical correlation $\rho(\mathbf{Z}_{1,j}^m, \mathbf{Z}_{k,j}^m)$ in component $j$, $j=1,\ldots,J$. The block-diagonal structure implies that
\[
\rho_{k,m}=\max\{\rho_{k,m,j},~j=1,\ldots,J\}.
\]

\end{remark}
\begin{proposition}\label{Pro:easy bound conseq}
Assumption A3 holds if $b_n=o(n)$ and
\begin{equation}\label{eq:block cond}
 \sum_{k=0}^{n}\min \left\{ \frac{b_n}{\lambda_{b_n+l}} M_\gamma(k), 1\right \}=o(n).
\end{equation}
\end{proposition}
\begin{proof}
In view of Lemma \ref{Lem:naive bound}, we have
\[
\sum_{k=0}^n \rho_{k,b_n+l}\le (b_n+l) + \sum_{k=b_n+l}^n \mathrm{min}\left\{J b_n \frac{M(k-b_n-l)}{\lambda_{b_n+l}}, ~1 \right\}  =o(n)
\]
since $b_n=o(n)$. Hence Assumption A3 holds.
\end{proof}

\noindent\textbf{Implications of Proposition \ref{Pro:easy bound conseq}}.

\medskip

We  discuss here the implications of  Condition (\ref{eq:block cond}) in various specific situations.
This discussion is restricted to the case $J=1$ which is of most interest. This discussion can be easily extended to the case of independent components  via the observation made in Remark \ref{Rem:indep}.
Let $c,C>0$ be  generic constants whose value can change from expression to expression. The notation $a\asymp b$ means $cb\le  a\le Cb $ for some $0<c<C$. Assume throughout that the covariance $\gamma(n)\rightarrow 0$ and $b_n=o(n)$ as $n\rightarrow\infty$. We distinguish two cases:  $\mathrm{ess\,inf}_\omega f(\omega)>0$ and  $\mathrm{ess\,inf}_\omega f(\omega)=0$.

\medskip

\noindent\fbox{\textbf{1. }Assume first $\mathrm{ess\,inf}_\omega f(\omega)>0$.}

\medskip
 In view of (\ref{eq:lambda lower bound}), the minimum eigenvalue $\lambda_m$ is bounded below away from zero, and hence Condition (\ref{eq:block cond}) holds if
\begin{equation}\label{eq:m_f>0}
b_n\sum_{k=0}^{n} M_\gamma(k)=o(n),
\end{equation}
where $M_\gamma(k)$ is expressed as (\ref{eq:M_gamma one dim}).
Consider the case $\sum_{k=0}^{\infty} M_\gamma(k)<\infty$, which implies the typical \emph{short-range dependence} condition: $\sum_{k=1}^\infty |\gamma(k)|< \sum_{k=0}^{\infty} M_\gamma(k) <\infty$. Then (\ref{eq:m_f>0}) reduces to  $b_n=o(n)$.  We get in particular: 
\begin{corollary}\label{cor:SRD sub}
Suppose that  $\mathrm{ess\,inf}_\omega f(\omega)>0$, and $|\gamma(n)|\le d_n$, where $d_n$ is non-increasing and summable (typically, $d_n=cn^{-\beta}$ for some constant $c>0$ and $\beta>1$). If $b_n=o(n)$, then Assumption A3 holds. 
\end{corollary}
\begin{proof}
$|\gamma(k)|\le d_k$ implies   $M_\gamma(k)\le d_k$, and hence $\sum_{k=0}^{\infty} M_\gamma(k) <\infty$.
\end{proof}

Consider now the situation relevant to \emph{long-range dependence}:
\begin{equation}\label{eq:long memory cov}
\gamma(k)= k^{2H-2}L(k),\qquad 1/2<H<1,
\end{equation}
where $L(k)$ is a slowly varying function at infinity.
 By
Theorem 1.5.3 of \citet{bingham:goldie:teugels:1989:regular}, Condition (\ref{eq:long memory cov}) implies that $M_\gamma(k)\sim    k^{2H-2}L(k)$, which  entails that $\sum_{k=0}^{n} M_\gamma(k)\le c  n^{2H-1}L(n)$. Thus (\ref{eq:m_f>0}) holds if
\begin{equation}\label{eq:B_N long memory}
b_n=o(n^{2-2H}L(n)^{-1}).
\end{equation}
So, the larger $H$, the smaller the block size $b_n$.
\begin{corollary}\label{cor:LRD sub}
Suppose that  $\mathrm{ess\,inf}_\omega f(\omega)>0$, and $|\gamma(n)|\le n^{2H-2}L(n)$, where $1/2<H<1$ and $L$ is slowly varying. If $b_n=o(n^{2-2H}L(n)^{-1})$, then Assumption A3 holds. 
\end{corollary}
The case $|\gamma(k)|\le k^{2H-2}L(k)$ 
also encompasses the seasonal long memory situations (see, e.g., \citet{haye:viano:2003:limit}),  where $\gamma(k)$ oscillates within a power-law envelope.

In the long-range dependent case, \cite{betken:Wendler:2015:subsampling}  obtained recently a bound for $\rho_{k,m}$ in (\ref{eq:spec radius}) using a result of \citet{adenstedt:1974:large}  under some additional assumptions. Their bound allows (\ref{eq:block cond gen}) to hold under the block size condition
\begin{equation}\label{eq:BW b_n}
b_n=o(n^{3/2-H-\epsilon})
\end{equation}
with arbitrarily small $\epsilon>0$. The condition (\ref{eq:BW b_n}) is better than  (\ref{eq:B_N long memory}) for each $H$,  and  $b_n =O(n^{1/2})$ is always allowed.

We have also seen  that if the model satisfies the assumptions of Proposition \ref{Pro:new bound}, one can choose
\[
b_n=o(n),
\]
irrespective of the value of $H\in (1/2,1)$.

\medskip

\noindent\fbox{\textbf{2. }Assume now $\mathrm{ess\,inf}_\omega f(\omega)=0$.}

\medskip
As mentioned in (\ref{eq:lambda lower bound}),
the smallest covariance eigenvalue $\lambda_m$  converges to $\mathrm{ess\,inf}_\omega f(\omega)=0$ as $m \rightarrow \infty$.
The rate of convergence
has been investigated by a number of authors. See,  e.g., \citet{kac:1953:eigen}, \citet{pourahmadi:1988:remarks}, \citet{serra:1998:extreme}, \citet{tilli:2003:universal} and \citet{simonenko:2005:asymptotic}. It involves the order of the zeros of $f(\omega)$. We say $f(\omega)$ has a \emph{zero of order} $\nu>0$ at $\omega=\omega_0$ if $f(\omega)\asymp|\omega-\omega_0|^{\nu}$.
Roughly speaking, the rate  at which $\lambda_m$ converges to zero follows the highest order of the zeros of $f(\omega)$, and the  rate of convergence to zero cannot be faster than exponential:
\begin{equation}\label{eq:universal lambda bound}
\lambda_m\ge e^{-cm}
\end{equation}
for some $c>0$ (see \citet{pourahmadi:1988:remarks} and \citet{tilli:2003:universal}).
Let us  focus on the situation where $f(\omega)$ has a finite number of zeros of polynomial orders. Specifically,
 suppose that $f(\omega)$ has zeros of order $\nu_1,\ldots,\nu_p$ at $p$ distinct points $\omega_1,\ldots,\omega_p$, and $f(\omega)$ stays positive outside arbitrary neighborhoods of $\omega_1,\ldots,\omega_p$. Then by Theorem 2.2 of \citet{simonenko:2005:asymptotic}, one has
$\lambda_m\asymp m^{-\nu}$ where
\[
\nu=\max(\nu_1,\ldots,\nu_p).\]
Therefore, 
\[\lambda_{b_n+l}\asymp (b_n+l)^{-\nu}\asymp b_n^{-\nu}
\] 
and since $M_\gamma(k)$ is non-increasing, we have
\begin{equation}\label{eq:bound zero case}
\sum_{k=0}^{n}\min \left\{ \frac{b_n}{\lambda_{b_n+l}} M_\gamma(k), 1\right \}\le \sum_{k=0}^{p_n} 1 + Cb_n^{1+\nu}\sum_{k=p_n+1}^n M_\gamma(k)\le  C\left(p_n+n b_n^{1+\nu} M_{\gamma}(p_n) \right).
\end{equation}
To satisfy (\ref{eq:block cond}), we need the last expression in (\ref{eq:bound zero case}) to be of order $o(n)$.
This will be so if   as $n\rightarrow\infty$, $p_n=o(n)$, and
\begin{equation}\label{eq:b_n remark restr}
b_n=o\left([M_{\gamma}\left(p_n\right)]^{-1/(1+\nu)}\right).
\end{equation}
To get the weakest restriction on $b_n$,  let in addition $p_n$ grow fast enough  so that $n/p_n=o(n^{\delta})$ for any $\delta>0$ (e.g., choose $n/p_n \asymp\log n$).
We have the following two typical cases:
\begin{itemize}
\item $M_\gamma(k)=O(e^{-k})$ decays exponentially. In this case, $[M_{\gamma}\left(p_n\right)]^{-1/(1+\nu)}=O(e^{p_n/(1+\nu)} )$, so the condition (\ref{eq:b_n remark restr}) is certainly satisfied when $b_n=o(n)$. Hence Assumption A3  holds with $b_n=o(n)$;
\item $M_\gamma(k)=O(k^{-\beta})$, $\beta>0$. In this case,  (\ref{eq:block cond}) holds when
\begin{equation}\label{eq:block cond nu beta}
b_n=o(n^{\beta/(1+\nu)-\epsilon})
\end{equation}
for arbitrarily small $\epsilon>0$. So the worst case is when $\beta$ is close to $0$ and $\nu$ is large.

A nice example involving both $\nu$ an $\beta$ is when $Z(n)$ is \emph{anti-persistent} (also called negative memory), e.g., the fractional Gaussian noise (the increments of fractional Brownian motion) with $H<1/2$, and FARIMA$(p,d,q)$ with $d=H-1/2$ so that $-1/2<d<0$. In this case, we have $\beta=2-2H$   and $\nu=1-2H$ in (\ref{eq:block cond nu beta}), and hence  (\ref{eq:block cond}) holds with $b_n=o(n^{1-\epsilon})$. Therefore: 

\end{itemize}

\begin{corollary}   
Suppose that $\{Z_n\}$ is fractional Gaussian noise with $H<1/2$ or  FARIMA$(p,d,q)$ with $-1/2<d<0$. If
$b_n=o(n^{1-\epsilon})
$ for  $\epsilon>0$ arbitrarily small, then  Assumption A3 holds.
\end{corollary}

\begin{remark}
We also mention that in  \citet{Zhang:Ho:Wendler:Wu:2013} which studies non-self-normalized block sampling for sample mean, the condition $b_n=o(n^{1-\epsilon})$ for arbitrarily small $\epsilon>0$ is shown to suffice for consistency.
The framework in their paper  assumes $\{X_i\}$ to be a univariate nonlinear transform of linear \emph{non-Gaussian} processes.
But it is not clear how to adapt their proof to a setting involving the self-normalization considered here.
\end{remark}

\subsection{Strong mixing case}

Given a stationary process $\{X_i\}$,  let $\mathcal{F}_{a}^b$ be the $\sigma$-field generated by $X_a,\ldots, X_b$,  where $ -\infty\le a\le b\le +\infty$. Recall that the strong mixing  (or $\alpha$-mixing) coefficient is defined as
\begin{equation}\label{eq:mixing coef}
\alpha(k)=\sup\left\{|P(A)P(B)-P(A\cap B)|, ~A\in \mathcal{F}_{-\infty}^0,B\in \mathcal{F}_{k}^{\infty} \right\}.
\end{equation}
Note that $0\le \alpha(k)\le 1$.
The process
$\{X_i\}$ is said to be \emph{strong mixing} if
\[
\lim_{k\rightarrow+\infty}\alpha(k)=0.
\]
We refer the reader to \citet{Bradley:2007} for more details. We shall use the following inequality which can be found in Lemma A.0.2  of \citet{politis:1999:subsampling}.
\begin{lemma}\label{Lem:mixing}
If $U\in \mathcal{F}_{-\infty}^0$ and $V\in \mathcal{F}_k^\infty$, and $ 0\le U, V\le 1$ almost surely, then
\[
|\Cov(U,V)|\le \alpha(k)\le 1.
\]
\end{lemma}

We shall assume:
\bigskip

\fbox{\parbox{\textwidth}{
\begin{enumerate}[\textbf{{B}}1.]
\item $\{X_i\}$ is a strong mixing stationary process with mean $\mu=\E X_i$.
\item We have the weak convergence in $D[0,1]$ endowed with $M_2$ topology of the partial sum:
\[
\left\{\frac{1}{n^{H}\ell(n)}(S_{\lfloor nt \rfloor} - n\mu),\ 0 \leq t \leq 1\right\} \Rightarrow \left\{Y(t),\ 0 \leq t \leq 1\right\},
\]
for some nonzero $H$-sssi process $Y(t)$, where $0<H<1$ and $\ell(\cdot)$ is a slowly varying function.
\item  The block size $b_n\rightarrow\infty$ and $b_n=o(n)$ as $n\rightarrow\infty$.
\end{enumerate}
}}

\bigskip

The following theorem establishes the consistency of the self-normalized block sampling under the strong mixing framework.
\begin{theorem}\label{Thm:main mixing}
The conclusions of Theorem \ref{Thm:main gaussian} and of  Corollary \ref{Cor:gaussian}  hold under Assumptions B1--B3.
\end{theorem}
\begin{proof}
The structure of the proof and many details are similar to those of Theorem \ref{Thm:main gaussian}.
 We only highlight the key differences. See also  \citet{politis:1999:subsampling} or \citet{sherman:carlstein:1996:replicate}.

In Step 1, we again need to show (\ref{eq:reduced goal}). The term $[P(T_{i,b_n}^*\le x)-P(T\le x)]^2\rightarrow 0$ as before. We need to establish $\Var[\widehat{F}_{n,b_n}^*(x)]\rightarrow 0$.  We still have the bound (\ref{eq:var bound}).

In view of Lemma \ref{Lem:mixing}, one has that,
\begin{align*}
\big|\Cov[\I\{T_{1,b_n}^*\le x\},\I\{T_{k+1,b_n}^*\le x\}]\big|\le
\begin{cases}
1 \quad&\text{ if }  k<b_n, \\
 \alpha(k-b_n+1),  \quad& \text{ if } k\ge b_n;
\end{cases}
\end{align*}
where  $\alpha(\cdot)$ is the mixing coefficient in (\ref{eq:mixing coef}). Hence from (\ref{eq:var bound}), we have
\begin{align}\label{eq:var bound mixing}
\Var[\widehat{F}_{n,b_n}^*(x)]&\le \frac{2}{n-b_n+1} \left(\sum_{k=0}^{b_n-1} \big|\Cov\big[\I\{T_{1,b_n}^*\le x\},\I\{T_{k+1,b_n}^*\le x\}\big]\big|\notag \right.  
\\  &~~ \left.+\sum_{k=b_n}^{n} \big|\Cov\big[\I\{T_{1,b_n}^*\le x\},\I\{T_{k+1,b_n}^*\le x\}\big]\big| \right)\notag
\\
&\le
 \frac{2}{(n-b_n+1)}\left[ b_n+ \sum_{k=b_n}^{n} \alpha(k-b_n+1)\right]\notag
 \\&=
 \frac{2b_n}{(n-b_n+1)}+ \frac{2}{(n-b_n+1)} \sum_{k=1}^{n-b_n+1} \alpha(k) ,
\end{align}
which converges to zero as $n\rightarrow\infty$, because $b_n=o(n)$ by Assumption B3 , and  $\alpha(k)\rightarrow 0$ as $k\rightarrow\infty$ by Assumption B1 and by applying a Ces\`aro summation.
Hence (\ref{eq:reduced goal}) is proved.

Step 2 and 3 proceed exactly as the proof of Theorem \ref{Thm:main gaussian}. The argument in the proof of Corollary \ref{Cor:gaussian} shows that the conclusion of that corollary continues to hold under Assumptions B1--B3.
\end{proof}

\begin{remark}
In view of \citet{Shao:2010}, the self-normalized block sampling method considered in this paper may be extended to more general statistics beyond the sample mean. There are two aspects to consider, self-normalization and block sampling.  For the \emph{self-normalization} aspect to work,  the general statistics   needs to be approximately linear, namely, it admits a functional Taylor expansion in the sense of (2) in \citet{Shao:2010}. In this case,   Assumption A2 or B2 needs to be replaced by a modified version of Assumption 1 of \citet{Shao:2010}. Furthermore, the remainder term in the aforementioned functional Taylor expansion has to satisfy a negligibility condition (see Assumption 2 of \citet{Shao:2010} or Assumption II of \citet{Shao:2015}). Validating these conditions for particular statistics (e.g., sample quantiles) and particular models (e.g., the Gaussian subordination model in Assumption A1) may be considered in future work.  The  \emph{block sampling} aspect is likely to continue to be valid, since as shown in the proofs of Theorem \ref{Thm:main gaussian} and \ref{Thm:main mixing}, the key is to have a bound on the between-block correlation,  as the one in Proposition \ref{Pro:new bound} in the long-memory Gaussian subordination framework,  or as in Lemma \ref{Lem:mixing} in the strong mixing framework. 
\end{remark}

\section{Examples}\label{sec:example}
The first two examples of models concern Assumptions A1--A3.  They both involve a phase transition.
\begin{example}\label{eg:gauss light}
Suppose that
\[X_i=G(Z_i)=Z_i^2,\]
where $\{Z_i\}$ is a standardized stationary Gaussian process with covariance $\gamma(n)=n^{2d-1}L(n)$, with $d\in (0,1/2)$, and $L(n)$ is a positive slowly varying function.  Then Assumption A1 is satisfied. Moreover,
by  \citet{Taqqu:1975} in the case $d<1/4$ and \citet{breuer:major:1983:central}  and \citet{chambers:slud:1989:central} in the case $d>1/4$,  Assumption A2
holds with the following dichotomy:
\[
\begin{cases}
H=1/2, ~\ell(n)=1, ~ Y(t)= \sigma B(t)  \quad &\text{ if }d<1/4;\\
H=2d, ~\ell(n)=L(n), ~  Y(t)=  c_H Z_{2,H}(t) \quad &\text{ if }d>1/4,
\end{cases}
\]
where $\sigma^2=\sum_{n}\Cov[X(n),X(0)]$, $c_H$ is a positive constant, $B(t)$ is the standard Brownian motion and $Z_{2,H}$ is the standard Rosenblatt process (second-order Hermite process). Assume in addition that the assumptions for $\{Z_i\}$ in Proposition \ref{Pro:new bound} hold.  Then one can choose a block size $b_n=o(n)$ to satisfy Assumption A3. Hence Theorem \ref{Thm:main gaussian} and Corollary \ref{Cor:gaussian} hold. Without the additional assumptions in Proposition \ref{Pro:new bound}, Assumption A3 is guaranteed  at least by the choice  $b_n=o(n^{1-2d}L(n)^{-1})$   in view of (\ref{eq:B_N long memory}).
\end{example}

\medskip
\begin{example}
Let $F_\alpha$ be the cdf of $t_\alpha$ distribution with $1<\alpha<2$, so that it has finite mean but infinite variance. Let $\Phi$ be the cdf of a standard normal. 
Suppose that
\[X_i=F_\alpha^{-1}(\Phi(Z_i)),\]
where $\{Z_i\}$ is a standardized stationary Gaussian process with covariance $\gamma(n)=n^{2d-1}L(n)$,  $d\in (0,1/2)$, and $L(n)$ is a positive slowly varying function. The marginal distribution of $\{X_i\}$ is a $t_\alpha$.   Then Assumption A1 is satisfied. By Sly and Heyde (2008), Assumption A2
holds with the following dichotomy (for $0<d<1/2$, $1<\alpha<2$):
\[
\begin{cases}
H=1/\alpha, ~\ell(n)=1, ~ Y(t)=  c_1  L_{\alpha}(t)\quad &\text{ if }d+1/2<1/\alpha;\\
H=d+1/2, ~\ell(n)=L(n), ~  Y(t)=  c_2 B_{H}(t) \quad &\text{ if } d+1/2>1/\alpha,
\end{cases}
\]
where $c_1$ and $c_2$ are  positive constants, $L_{\alpha}(t)$ is a symmetric $\alpha$-stable L\'evy process, and $B_{H}(t)$ is a standard fractional Brownian motion. Assume in addition that the assumptions for $\{Z_i\}$ in Proposition \ref{Pro:new bound} hold.  This  will be the case if $\{Z_i\}$ is fractional Gaussian noise or FARIMA$(p,d,q)$. Then   $b_n=o(n)$ implies (\ref{eq:block cond gen}). Hence Theorem \ref{Thm:main gaussian} and Corollary \ref{Cor:gaussian} hold. Without the additional assumptions in Proposition \ref{Pro:new bound}, Assumption A3 is guaranteed  at least by the choice  $b_n=o(n^{1-2d}L(n)^{-1})$   in view of (\ref{eq:B_N long memory}).
\end{example}

\begin{example}\label{eg:gauss heavy}
Consider the following long-memory stochastic duration (LMSD) model (for modeling inter-trade duration, see \citet{deo:hsieh:hurvich:2010:long}):
\[X_i = \xi_i\exp(Z_i) ,\]
where  $\{\xi_i\}$ are i.i.d.\ positive random variables satisfying
$P(\xi_i>x)\sim A x^{-\alpha}$ as $x\rightarrow\infty$, $A>0$, $\alpha\in (1,2)$,  $Z_i$ is a Gaussian linear process  $Z_i=\sum_{j=1}^\infty j^{d-1}l(j)\epsilon_{i-j}$ with $d\in (0,1/2)$, $l(j)$ a positive and slowly varying function, $\{\epsilon_i\}$ i.i.d.\ centered Gaussian,  and $\{\epsilon_i\}$ is independent of $\{\xi_i\}$. Note that  $\mu=\E X_i>0$. The model has the interesting feature that  although $\E X_i^2=\infty$, it has the following   finite covariance  for $h\neq 0$, namely,
\[
\Cov[X_i,X_{i+h}]=\Cov[\exp(Z_0),  \exp(Z_h)] \mu_\xi^2 \sim c h^{2d-1} l^2(h),
\]
as $h\rightarrow\infty$,
where $\mu_\xi=\E \xi_i$, and we have used the fact that the exponential function has Hermite rank 1 (see \citet{Taqqu:1975}).
To satisfy Assumption A1, one can rewrite the model as
\[X_i=g(Z_i')\exp(Z_i),\]
where $\{Z_i'\}$ are i.i.d.\ standard Gaussian with $g$  chosen such that $g(Z_i')$ is equal in distribution to $\xi_i$. This makes the model satisfy Assumption A1 with $J=2$, $l=0$, $\mathbf{Z}_i=(Z_i',Z_i)$ and $G(x_1,x_2)=g(x_1)\exp(x_2) $.
By (4.100) and (4.101) of \cite{beran:2013:long}, Assumption A2
holds with the following dichotomy:
\[
\begin{cases}
H=1/\alpha, ~\ell(n)=1, ~ Y(t)=  c_\alpha  L_{\alpha,1,1}(t)\quad &\text{ if }d+1/2<1/\alpha;\\
H=d+1/2, ~\ell(n)=l^2(n), ~  Y(t)=  c_d B_H(t) \quad &\text{ if } d+1/2>1/\alpha,
\end{cases}
\]
where $c_\alpha$, $c_{d}$ are positive constants, $L_{\alpha,1,1}(t)$ is an $\alpha$-stable  L\'evy process with skewness $\beta=1$ (see (\ref{eqn:SntSRDStable})), and $B_H(t)$ is the standard fractional Brownian motion. If in addition, the assumptions for $\{Z_i\}$ in Proposition \ref{Pro:new bound} hold,  then Assumption A3 is satisfied if $b_n=o(n)$. Hence Theorem \ref{Thm:main gaussian} and Corollary \ref{Cor:gaussian} hold. Without the additional assumptions in Proposition \ref{Pro:new bound}, Assumption A3 is at least satisfied if $b_n=o(n^{1-2d}l(n)^{-2})$ (see (\ref{eq:B_N long memory}) and Remark \ref{Rem:indep}).
\end{example}

\medskip
\begin{remark}
Consider the non-centered stochastic volatility model $X_i=\sigma_ig(Z_i)+\mu$ in \citet{jach:2012:subsampling}, where $\sigma_i$ and $g(Z_i)$ are independent, $\sigma_i$ is i.i.d.\ with heavy tails and $\{Z_i\}$ is Gaussian with long-range dependence and $g$ has Hermite rank one. This model can be  similarly embedded into Assumption A1. However, as far as we know, the functional convergence\footnote{The weak convergence  assumed in Assumption A2 allowed us to take advantage of Lemma \ref{Lem:int cont} in order to establish Lemma \ref{Lem:cont mapping}.} needed in Assumption A2   has not been established  (only the marginal convergence was established in \citet{jach:2012:subsampling}).  Assumption A2 for this model is, nevertheless, expected to hold in view of its similarity\footnote{Both \citet{jach:2012:subsampling} and \citet{kulik:soulier:2012:limit} treated stochastic volatility models of the form $X_i=L_iH_i$ (for limit theorems it does not matter whether a level is added or not), where $L_i$ has finite variance and is long-range dependent, while $H_i$  has infinite variance and is i.i.d.. The difference between the two papers is that in \citet{jach:2012:subsampling} $L_i$ is  centered and $H_i$ is not, while in \citet{kulik:soulier:2012:limit} $H_i$ is  centered and $L_i$ is not.} to the model treated in \citet{kulik:soulier:2012:limit}, Theorem 4.1 (see also Theorem 4.19 of \cite{beran:2013:long}).  Checking Assumption A2 in details is outside the scope of the current paper. Assumption A3 is dealt with as in Example \ref{eg:gauss heavy}.


Nevertheless,  the consistency of the self-normalized block sampling  in \citet{jach:2012:subsampling} can be shown to hold under our A1 and A3 framework. This is done by adopting the normalization of \citet{jach:2012:subsampling}, with A2 replaced by marginal convergence involving partial sums and sample covariances\footnote{More precisely, convergence in distribution of a 3-dimensional vector specified in Theorem 3 of \citet{jach:2012:subsampling}.},
and to ensure A3, by  assuming $b_n=o(n)$ and that $\{Z_i\}$ is a long-range dependent sequence satisfying the assumptions of Proposition \ref{Pro:new bound}.
\end{remark}

We now give two examples with \emph{strong mixing}. The first involves a nonlinear time series and the second involves heavy tails.

\begin{example}\label{eg:mix light}
Suppose that
\begin{equation}\label{eqn:TAR}
X_i = \rho |X_{i-1}| + \epsilon_i, \quad 0<\rho<1,
\end{equation}
where $\epsilon_i$'s are i.i.d.\ standard Gaussian. Thus $\{X_i\}$ follows a threshold autoregressive model (\citet{Tong:1990}).
The Markov process $\{X_i\}$ is strong mixing because it is ergodic\footnote{that is, the Markov chain is irreducible aperiodic and positive recurrent (see \citet{tweedie:1975:sufficient}).} (see  \citet{petruccelli:woolford:1984:threshold}, Theorem 2.1, or \citet{doukhan:1994:mixing} p.103), and hence Condition B1 holds.
The conditions of Theorem 3(ii) of  \citet{Wu:2005} are satisfied\footnote{In the terminology of \citet{Wu:2005},  $R(x,\epsilon)=\rho|x|+\epsilon$, $L_\epsilon=\rho$, $\delta_p(n)=O(n^r)$ for some $0<r<1$, so that $\sum_{n=0}^\infty n \delta_p(n)<\infty$, implying Theorem 3(ii).} and therefore Condition B2 holds with $H=1/2$, $\ell(n)=1$ and $Y(t)=\sigma B(t)$, where $\sigma^2=\sum_{n}\gamma(n)>0$ and $B(t)$ is standard Brownian motion. Condition B3 holds for any block size $b_n=o(n)$. Therefore, Theorem \ref{Thm:main mixing} holds.
\end{example}

In the following example, both Assumptions A1--A3 and B1--B3 hold.

\begin{example}\label{Eg:mix heavy}
Consider the MA(1) model
\[
X_i=\epsilon_i+a\epsilon_{i-1},
\]
where $a\ge 0$ and $\{\epsilon_i\}$ are i.i.d.. Assume that $\E \epsilon_i=0$, $\E \epsilon_i^2=\infty$, and $\epsilon_i$  is in the domain of attraction of a stable distribution with an index $\alpha\in (1,2)$. Let $b_n=o(n)$. By choosing appropriate transforms, we can express $\epsilon_i$ as function of Gaussian. Therefore Assumption A1 holds. Assumption B1 holds because $\{X_i\}$ is $2$-dependent.
 By  Theorem 2' of \citet{avram:taqqu:1992:weak},  Assumptions A2 and B2 hold with $H=1/\alpha$, some slowly varying function $\ell(n)$, and $Y(t)$ is an $\alpha$-stable L\'evy process. Also A3 holds with any $b_n=o(n)$ since $\rho_{k,m}=0$ when $k\ge m+2$. Therefore, both assumptions  A1--A3 and B1--B3 hold in this case.
\end{example}

\section{Monte Carlo Simulations}\label{sec:numericalexperiments}
We shall  carry out here Monte Carlo simulations to examine the finite-sample performance of the self-normalized block sampling (SNBS) method and make a comparison with the recent result of \citet{Zhang:Ho:Wendler:Wu:2013}. Instead of resorting to self-normalization, the method of \citet{Zhang:Ho:Wendler:Wu:2013} exploits the regularly varying property of the asymptotic variance to avoid the problem of estimating the nuisance Hurst index. We first consider the case with Gaussian subordination. For this, let
\begin{equation}\label{eqn:Xisimulation}
X_i = K(Z_i),\quad Z_i = \sum_{j=0}^\infty a_j \epsilon_{i-j},\quad i = 1,\ldots,n,
\end{equation}
where $K(\cdot)$ is a possibly nonlinear transformation and $\{\epsilon_k\}$ are i.i.d.\  standard normal random variables\footnote{To generate the process, we use the approximation $Z_i \approx \sum_{j=0}^{\lfloor n^{3/2} \rfloor - 1} a_j \epsilon_{i-j}$ in our simulation, and the fast Fourier transform (FFT) as mentioned in \citet{Wu:Michailidis:Zhang:2004} is implemented to facilitate the computation. Note that the cutoff $n^{3/2}$ is much greater than the sample size $n$.}. We consider the following configurations for (\ref{eqn:Xisimulation}):
\begin{enumerate}
\item[(a)] $K(x) = x$ and $a_j = (1+j)^{d-1}$, $j \geq 0$;
\item[(b)] $K(x) = x^2$ and $a_j = (1+j)^{d-1}$, $j \geq 0$;
\item[(c)] $K(x) = \Phi_t^{-1}[\Phi_N\{(\sum_{j=0}^\infty a_j^2)^{-1/2}x\}]$ and $a_j = (1+j)^{d-1}$, $j \geq 0$,
\end{enumerate}
where $\Phi_N$ is the CDF of the standard normal and $\Phi_t$ is the CDF of the Student's $t$-distribution with degree of freedom $1.5$, whose tail probability decays like $|x|^{-3/2}$ as $|x|\rightarrow\infty$ so that it has infinite variance but finite mean.

Case (a) represents the Gaussian linear process which has been extensively used in the literature for modeling time series data.  It has long-range dependence if $0<d<1/2$.  We let $d \in \{0.25,-1\}$. The choice $d=0.25$ corresponds to long-range dependence (LRD) and the choice $d=-1$ corresponds to short-range dependence (SRD).

Case (b) involves an additional nonlinear transformation and now $\{X_i\}$ is LRD if $0.25<d<0.5$. We let $d \in \{0.4,0.2,-1\}$.  When $d=0.4$, both $\{Z_i\}$ and $\{X_i\}$ have LRD (the limit for $\{X_i\}$ is the Rosenblatt process); when $d=0.2$, $\{Z_i\}$ has LRD and $\{X_i\}$ has SRD (the limit for $\{X_i\}$ is Brownian motion); when $d=-1$, both $\{Z_i\}$ and $\{X_i\}$  have SRD (the limit for $\{X_i\}$ is Brownian motion). See for example \citet{Wu:2006} and \citet{Zhang:Ho:Wendler:Wu:2013}.

Case (c) corresponds to a process $\{X_i\}$ with marginal distribution $t$ with $1.5$ degrees of freedom and hence with infinite variance. We let $d \in \{0.4,0.2,-1\}$. When $d=0.4$ and  $d=0.2$, both $\{Z_i\}$ and $\{X_i\}$ have LRD (the limit for $\{X_i\}$ is the fractional Brownian motion); when $d=-1$, both $\{Z_i\}$ and $\{X_i\}$ have SRD (the limit for $\{X_i\}$ is symmetric $(3/2)$-stable L\'evy motion). See \citet{sly:heyde:2008:nonstandard} for the boundary between SRD and LRD in the heavy tail case.
We also consider the situation with a non-constant slowly varying function, where we let $a_j = (1+j)^{d-1} \log(1+j)$, $j \geq 0$, and denote the corresponding cases by (a$^*$), (b$^*$) and (c$^*$), respectively.

We consider the problem of constructing the lower and upper one-sided confidence interval where the nominal level is taken as 90\%; see also \citet{Nordman:Lahiri:2005} and \citet{Zhang:Ho:Wendler:Wu:2013} for similar performance assessment of this type. Following \citet{Zhang:Ho:Wendler:Wu:2013}, we use throughout the block sizes  $b_n = \lfloor cn^{0.5} \rfloor$, $c \in \{0.5,1,2\}$. This  does not necessarily represent the optimal choice of $b_n$, but provides us with a spectrum of reasonable block sizes in our finite-sample simulations. For each realization we compute the self-normalized block sums and its empirical distribution function $\hat F_{n,b_n}$ as in (\ref{eq:F_N^*}).  Examples of realized $\hat F_{n,b_n}$ can be found in Figure \ref{fig:cdf} for models (a)--(c) with different choices of $d$.
Let $q_\alpha$ ($\alpha$=10\%) be the 10\%-quantile of $\hat F_{n,b_n}$, then the lower 90\% one-sided confidence interval can be constructed as
\begin{equation*}
\left(-\infty~,~\bar X_n - n^{-1}\left\{n^{-1}\sum_{k=1}^n (S_{1,k}-\frac{k}{n}S_{1,n})^2\right\}^{1/2} q_\alpha\right];
\end{equation*}
Similarly, if $q_{1-\alpha}$ ($1-\alpha$=90\%) denotes the $90$\%-quantile of $\hat{F}_{n,b_n}$, then the corresponding uppper $90$\% one-sided confidence interval is 
\begin{equation*}
\left[\bar X_n - n^{-1}\left\{n^{-1}\sum_{k=1}^n (S_{1,k}-\frac{k}{n}S_{1,n})^2\right\}^{1/2} q_{1-\alpha}~,~+\infty\right).
\end{equation*}
See (\ref{eq:CI}) for details.

In Tables \ref{tab:simulation} and \ref{tab:simulationlog}, we report the empirical coverage probabilities of the constructed confidence intervals based on 5000 realizations for each scenario\footnote{When evaluating the empirical coverage probability of the constructed confidence interval, we use the averaged mean of 1000 realizations as an approximation to the true mean.}. For example, 
Table \ref{tab:simulation} displays the following results of simulation. If $d=0.25$, $c=0.5$ and $n=100$, then the self-normalized block sampling (SNBS) simulation yielded the following:  the lower $90$\% confidence interval included the unknown mean $\mu$,  $88.3$\% of the times and the upper $90$\% confidence interval included the unknown mean $\mu$, $91.1$\% of the times.  We also report the results of the subsampling method of \citet{Zhang:Ho:Wendler:Wu:2013} for a comparison in the column ZHWW2013. Note that the method of \citet{Zhang:Ho:Wendler:Wu:2013} does not take advantage of the technique of self-normalization and therefore it requires an additional bandwidth to utilize the regularly varying property of the asymptotic variance.\footnote{In Tables \ref{tab:simulation} and \ref{tab:simulationlog}, we let the second bandwidth be $l_n = \lfloor  n^{0.9} \rfloor $ when using the method of \citet{Zhang:Ho:Wendler:Wu:2013}.  Many other choices are possible.  We also used $l_n = \lfloor 0.5 n^{0.9} \rfloor$ and obtained similar results.}.

\begin{figure}[!t]
\centering
\includegraphics[width = \textwidth]{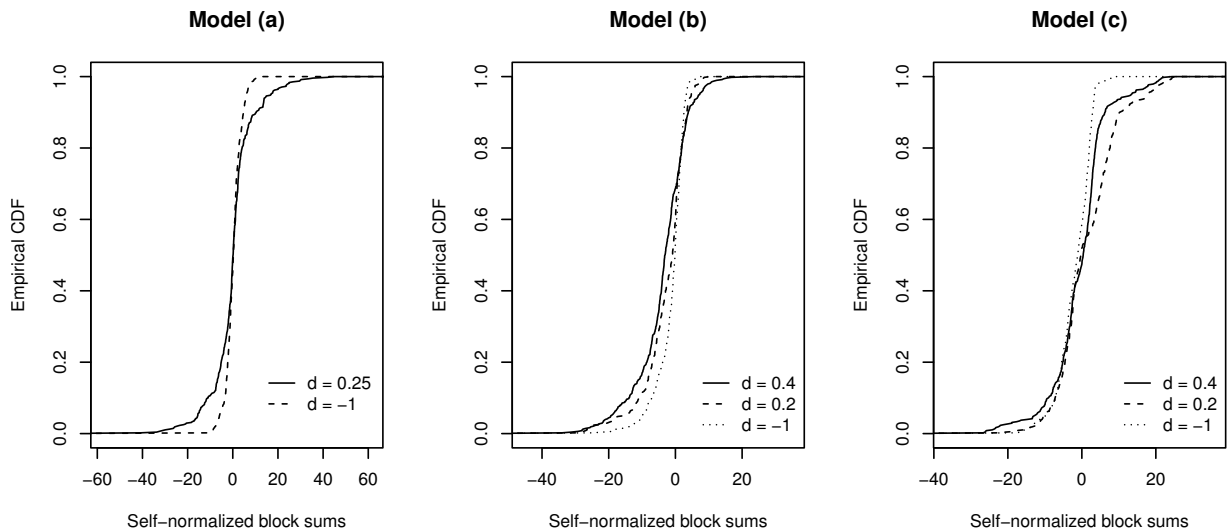}
\caption{Examples of realized $\hat F_{n,b_n}$ for models (a)--(c) with $n = 500$, $c = 1$ and different choices of $d$. The $x$-axis represents the self-normalized block sums, which have been appropriately centered and scaled. }\label{fig:cdf}
\end{figure}

It can be seen from Tables \ref{tab:simulation} and \ref{tab:simulationlog} that the method proposed in this paper performs reasonably well, as most of the empirical coverage probabilities are reasonably close to their nominal level of 90\%, except for situations with heavy tails where  deviations under small sample sizes are expected. However, the results seem to improve as the sample size increases from $n = 100$ to $n = 500$ and the performance is comparable to the method of \citet{Zhang:Ho:Wendler:Wu:2013}\footnote{The theoretical assumptions in \citet{Zhang:Ho:Wendler:Wu:2013} do not allow for infinite variance.}. Note that the choice of sample size $n = 100$ is considered to be challengingly small for inference of long-range dependent processes. Because of self-normalization, our method has the advantage over the one by \citet{Zhang:Ho:Wendler:Wu:2013} in  not requiring the choice of a second bandwidth.

\begin{table}[!t]\footnotesize
\centering
\begin{tabular}{ccccccc}
\hline \hline
 & & \multicolumn{2}{c}{n = 100} & & \multicolumn{2}{c}{n = 500} \\ \cline{3-4} \cline{6-7}
$d$ & $c$ & SNBS & ZHWW2013 & & SNBS & ZHWW2013 \\
\hline
 & & & \multicolumn{3}{c}{\em Model (a)} & \\
0.25 & 0.5 & (88.3, 91.1) & (86.8, 90.3) & & (92.2, 92.0) & (92.0, 91.5) \\
     & 1   & (86.1, 86.6) & (85.7, 85.3) & & (89.6, 91.2) & (89.3, 91.3) \\
     & 2   & (82.3, 83.7) & (81.0, 82.2) & & (87.5, 87.5) & (87.4, 87.2) \\
-1   & 0.5 & (93.5, 94.2) & (93.0, 92.9) & & (93.2, 93.1) & (92.9, 93.0) \\
     & 1   & (89.5, 90.7) & (89.0, 90.2) & & (91.4, 92.1) & (91.1, 91.7) \\
     & 2   & (87.1, 86.3) & (86.9, 85.6) & & (90.0, 89.0) & (89.9, 89.5) \\
 & & & \multicolumn{3}{c}{\em Model (b)} & \\
0.4  & 0.5 & (90.3, 95.7) & (89.2, 95.2) & & (93.2, 96.2) & (92.9, 95.6) \\
     & 1   & (84.7, 93.6) & (83.8, 92.7) & & (88.2, 94.8) & (88.4, 94.9) \\
     & 2   & (75.9, 91.8) & (75.3, 91.4) & & (84.3, 92.8) & (84.0, 92.9) \\
0.2  & 0.5 & (94.6, 95.8) & (94.0, 94.8) & & (95.7, 96.0) & (95.8, 95.6) \\
     & 1   & (88.8, 93.6) & (88.2, 93.3) & & (93.8, 93.6) & (93.7, 93.9) \\
     & 2   & (81.4, 91.5) & (80.3, 90.8) & & (89.4, 92.0) & (89.3, 91.9) \\
-1   & 0.5 & (97.6, 86.3) & (97.5, 85.5) & & (97.0, 86.0) & (97.0, 86.1) \\
     & 1   & (94.1, 84.2) & (93.5, 83.3) & & (94.5, 86.5) & (94.3, 86.5) \\
     & 2   & (87.2, 84.0) & (86.7, 83.6) & & (91.3, 86.6) & (91.2, 86.7) \\
 & & & \multicolumn{3}{c}{\em Model (c)} & \\
0.4  & 0.5 & (74.8, 84.4) & (72.5, 82.9) & & (82.2, 78.0) & (81.8, 77.1) \\
     & 1   & (78.0, 76.9) & (76.5, 75.8) & & (77.7, 79.3) & (76.9, 78.9) \\
     & 2   & (75.5, 73.4) & (74.8, 72.2) & & (74.6, 78.6) & (73.8, 78.4) \\
0.2  & 0.5 & (78.8, 81.4) & (76.7, 79.0) & & (80.8, 79.9) & (80.0, 79.6) \\
     & 1   & (77.0, 80.6) & (75.9, 79.6) & & (79.1, 80.8) & (78.7, 80.0) \\
     & 2   & (77.9, 74.8) & (76.6, 74.1) & & (81.1, 77.3) & (80.9, 76.3) \\
-1   & 0.5 & (82.3, 83.7) & (80.9, 82.2) & & (83.6, 85.3) & (83.3, 84.2) \\
     & 1   & (84.1, 80.0) & (83.2, 79.4) & & (81.6, 86.0) & (80.6, 85.5) \\
     & 2   & (87.4, 71.2) & (86.2, 70.3) & & (82.0, 82.9) & (81.7, 82.8) \\
\hline
\end{tabular}
\caption{\footnotesize Empirical coverage probabilities of lower and upper (paired in parentheses) one-sided 90\% confidence intervals with different combinations of the index $d$, sample size $n$ and block size $b_n = \lfloor cn^{0.5} \rfloor$ when $a_j = (1+j)^{d-1}$, $j \geq 0$.}\label{tab:simulation}
\end{table}

\begin{table}[!t]\footnotesize
\centering
\begin{tabular}{ccccccc}
\hline \hline
 & & \multicolumn{2}{c}{n = 100} & & \multicolumn{2}{c}{n = 500} \\ \cline{3-4} \cline{6-7}
$d$ & $c$ & SNBS & ZHWW2013 & & SNBS & ZHWW2013 \\
\hline
 & & & \multicolumn{3}{c}{\em Model (a$^*$)} & \\
0.25 & 0.5 & (87.8, 87.9) & (86.4, 86.5) & & (91.9, 91.6) & (92.0, 91.6) \\
     & 1   & (84.0, 84.2) & (82.7, 83.0) & & (90.4, 89.4) & (90.1, 88.9) \\
     & 2   & (78.0, 79.2) & (76.8, 78.6) & & (84.7, 85.1) & (84.4, 84.8) \\
-1   & 0.5 & (93.7, 93.6) & (93.1, 92.5) & & (94.0, 94.4) & (93.6, 94.4) \\
     & 1   & (90.9, 89.8) & (90.1, 88.6) & & (93.2, 91.9) & (92.8, 92.0) \\
     & 2   & (86.4, 86.4) & (85.9, 85.4) & & (90.6, 90.3) & (90.2, 90.1) \\
 & & & \multicolumn{3}{c}{\em Model (b$^*$)} & \\
0.4  & 0.5 & (84.7, 95.1) & (83.3, 94.3) & & (90.3, 98.0) & (90.4, 97.7) \\
     & 1   & (80.4, 92.3) & (79.4, 91.8) & & (86.2, 96.0) & (86.4, 96.0) \\
     & 2   & (71.7, 90.2) & (70.5, 89.6) & & (79.5, 93.6) & (79.6, 93.9) \\
0.2  & 0.5 & (89.3, 96.8) & (88.7, 96.3) & & (94.3, 97.7) & (94.4, 97.5) \\
     & 1   & (83.6, 93.9) & (83.0, 93.3) & & (90.8, 96.7) & (91.1, 96.7) \\
     & 2   & (77.7, 91.2) & (76.8, 90.4) & & (85.1, 95.6) & (85.1, 95.2) \\
-1   & 0.5 & (98.3, 86.3) & (97.9, 85.6) & & (97.1, 87.0) & (97.1, 87.1) \\
     & 1   & (93.1, 85.2) & (92.8, 84.6) & & (95.1, 85.6) & (94.7, 85.7) \\
     & 2   & (88.6, 84.1) & (87.9, 83.5) & & (92.2, 85.9) & (92.1, 85.5) \\
 & & & \multicolumn{3}{c}{\em Model (c$^*$)} & \\
0.4  & 0.5 & (86.3, 85.8) & (84.6, 83.7) & & (92.6, 88.0) & (92.3, 88.0) \\
     & 1   & (83.4, 77.7) & (82.3, 76.1) & & (87.3, 83.3) & (87.7, 83.5) \\
     & 2   & (74.9, 75.2) & (73.2, 74.0) & & (81.1, 81.7) & (81.1, 81.1) \\
0.2  & 0.5 & (83.0, 85.2) & (80.4, 83.3) & & (86.6, 84.6) & (86.8, 84.9) \\
     & 1   & (80.4, 80.5) & (79.5, 78.7) & & (84.4, 81.7) & (84.5, 80.7) \\
     & 2   & (77.9, 73.5) & (76.9, 72.9) & & (80.4, 78.8) & (80.9, 78.3) \\
-1   & 0.5 & (83.9, 83.1) & (82.3, 81.7) & & (88.5, 84.0) & (87.5, 83.0) \\
     & 1   & (80.6, 83.1) & (80.0, 81.7) & & (86.8, 83.4) & (85.9, 82.8) \\
     & 2   & (83.2, 76.7) & (82.2, 75.8) & & (85.8, 82.3) & (85.4, 81.5) \\
\hline
\end{tabular}
\caption{\footnotesize Empirical coverage probabilities of lower and upper (paired in parentheses) one-sided 90\% confidence intervals with different combinations of the index $d$, sample size $n$ and block size $b_n = \lfloor cn^{0.5} \rfloor$ when $a_j = (1+j)^{d-1}\log(1+j)$, $j \geq 0$.}\label{tab:simulationlog}
\end{table}

Finally, consider the strong mixing Example \ref{eg:mix light}, where
$
X_i = \rho |X_{i-1}| + \epsilon_i,
$
following the threshold autoregressive model \citep{Tong:1990}. The $\epsilon_i$'s are i.i.d.\ Gaussian.
 The results for $\rho = 0.5$ are summarized in Table \ref{tab:simulationTAR}. Observe that the method works quite well in this case as well.

\begin{table}[!t]\footnotesize
\centering
\begin{tabular}{cccccc}
\hline \hline
 & \multicolumn{2}{c}{n = 100} & & \multicolumn{2}{c}{n = 500} \\ \cline{2-3} \cline{5-6}
$c$ & SNBS & ZHWW2013 & & SNBS & ZHWW2013 \\
\hline
0.5 & (92.1, 94.3) & (91.7, 93.7) & & (93.2, 89.6) & (93.0, 89.4) \\
1   & (90.0, 88.9) & (88.8, 88.8) & & (91.0, 88.0) & (91.3, 88.5) \\
2   & (86.9, 84.7) & (86.1, 84.0) & & (89.9, 87.2) & (90.1, 87.3) \\
\hline
\end{tabular}
\caption{\footnotesize Empirical coverage probabilities of lower and upper (paired in parentheses) one-sided 90\% confidence intervals with the TAR model (\ref{eqn:TAR}) for different combinations of sample size $n$ and block size $b_n = \lfloor cn^{0.5} \rfloor$.}\label{tab:simulationTAR}
\end{table}

The \texttt R function implementing the method is available from the authors.

\medskip
\noindent\textbf{Acknowledgments} We thank the referees for their helpful comments and suggestions. We also thank Mamikon S. Ginovyan for his suggestions on an earlier version of the paper. The work was partially supported by the NSF grants DMS-1309009 and DMS-1461796 at Boston University.

\bibliographystyle{BibliographyStyleTing}
\setlength{\bibsep}{0mm}
\bibliography{BibliographyTing}

\medskip
\noindent Shuyang Bai~~ \textit{bsy9142@bu.edu}~~~~
Murad S. Taqqu ~~\textit{murad@bu.edu}~~~~
Ting Zhang~~\textit{tingz@bu.edu }\\
Department of Mathematics and Statistics\\
111 Cummington Mall\\
Boston, MA, 02215, US

\end{document}